\documentclass[11pt,reqno]{amsproc}
\usepackage[margin=1in]{geometry}
\usepackage{amsmath, amsthm, amssymb}
\usepackage[abbrev,lite,nobysame]{amsrefs}
\usepackage[usenames,dvipsnames]{color}
\usepackage{mathtools}
\usepackage{dsfont}
\usepackage[hidelinks]{hyperref}
\usepackage{todonotes}
\usepackage{mathscinet}
\usepackage{esint}

\mathtoolsset{showonlyrefs=true}

\DeclareMathOperator*{\Tr}{\mathrm{tr}}

\DeclareMathOperator{\tr}{\mathrm{tr}}

\newcommand{\R}{\mathbb{R}}
\newcommand{\T}{\mathbb{T}}

\newcommand{\Z}{\mathbb{Z}}

\usepackage{bbm}
\newcommand{\1}{\mathbbm{1}}

\newcommand{\N}{\mathbb{N}}
\newcommand{\ep}{\varepsilon}

\newcommand{\tensor}{\otimes}
\renewcommand{\S}{\mathbb{S}\,}
\newcommand{\E}{\mathbf{E}}
\renewcommand{\P}{\mathbf{P}}

\newcommand{\leqc}{\lesssim}

\newtheorem{theorem}{Theorem}[section]
\newtheorem{proposition}[theorem]{Proposition}

\newtheorem{lemma}[theorem]{Lemma}
\newtheorem*{lemma*}{Lemma}

\newtheorem{assumption}{Assumption}

\theoremstyle{definition}
\newtheorem{definition}[theorem]{Definition}
\newtheorem{remark}[theorem]{Remark}

\newcommand{\MMMsw}[1]{ [\Meas_{t,x,v}]_{*} }

\newcommand{\dee}{\mathrm{d}}
\newcommand{\ds}{\dee s}

\newcommand{\dx}{\dee x}

\newcommand{\dy}{\dee y}

\newcommand{\dz}{\dee z}

\newcommand{\dr}{\dee r}
\newcommand{\dn}{\dee n}

\newcommand{\Grad}{\nabla}

\newcommand{\norm}[1]{\left\| #1 \right\|}
\newcommand{\damp}{\alpha}
\newcommand{\da}{G} % some macro for the damping term for the velocity, change as you like

\newcommand{\Naturals}{\mathbb{N}}

\renewcommand{\norm}[1]{\|#1\|}
\newcommand{\abs}[1]{\left\vert#1\right\vert}
\newcommand{\set}[1]{\left\{#1\right\}}

\newcommand{\grad}{\nabla}

\newcommand{\Real}{\mathbb{R}}

\newcommand{\curl}{\mathrm{curl}\,}

\def\eps{\varepsilon}
\def\e{{\rm e}}
\def\dd{{\rm d}}

\def\EE{\mathbf{E}}
\def\TT{{\mathbb{T}}}

\numberwithin{equation}{section}

\begin{document}
\title[Sufficient conditions for dual cascade flux laws in 2d SNSE]{Sufficient conditions for dual cascade flux laws in the stochastic 2d Navier-Stokes equations}

\author[J. Bedrossian, M. Coti Zelati, S. Punshon-Smith, F. Weber]{Jacob Bedrossian, Michele Coti Zelati, Sam Punshon-Smith, Franziska Weber}

\address{Department of Mathematics, University of Maryland, College Park, MD 20742, USA}
\email{jacob@math.umd.edu}

\address{Department of Mathematics, Imperial College London, London, SW7 2AZ, UK}
\email{m.coti-zelati@imperial.ac.uk}

\address{Division of Applied Mathematics, Brown University, Providence, RI 02906, USA}
\email{punshs@brown.edu}

\address{Department of Mathematical Sciences, Carnegie Mellon University, Pittsburgh, PA 15213, USA}
\email{franzisw@andrew.cmu.edu}



\maketitle

\begin{abstract}
We provide sufficient conditions for mathematically rigorous proofs of the third order universal laws capturing the energy flux to large scales and enstrophy flux to small scales for statistically stationary, forced-dissipated 2d Navier-Stokes equations in the large-box limit. 
These laws should be regarded as 2d turbulence analogues of the $4/5$ law in 3d turbulence, predicting a constant flux of energy and enstrophy (respectively) through the two inertial ranges in the dual cascade of 2d turbulence.
Conditions implying only one of the two cascades are also obtained, 
as well as compactness criteria which show that the provided sufficient conditions are not far from being necessary.
The specific goal of the work is to provide the weakest characterizations of the ``0-th laws'' of 2d turbulence in order to make mathematically rigorous predictions consistent with experimental evidence. 
\end{abstract}

\setcounter{tocdepth}{1}
{\small\tableofcontents}

\section{Introduction}
In this paper, we provide sufficient conditions for the mathematically rigorous derivation of the third order universal laws for the statistics of stationary, forced-dissipated, 
two-dimensional turbulence. Two-dimensional turbulence is relevant in, for example, large scale atmospheric and oceanic dynamics where it provides a building-block for the more realistic models (see discussions in \cite{BE12} and the references therein).
The simplest mathematical setting  is via statistically stationary solutions of the incompressible Navier-Stokes equations subject to stochastic forcing and large-scale damping, written in velocity form as
\begin{equation}\label{eq:NSE}%\tag{\textsc{NSE}}
\begin{cases}
\partial_t u + (u\cdot\nabla) u + \nabla p = \nu \Delta u - \alpha (-\Delta)^{-2\gamma}u+ \partial_t W^\lambda, \\
 \nabla\cdot u = 0.
\end{cases}
\end{equation}
Here, the equations are posed on a periodic box $\T^2_\lambda=[0,\lambda)^2$ of size $\lambda>0$; the parameter $\nu>0$ plays the role of the inverse Reynolds number,
while $\alpha>0$ measures the strength of the generalized linear Ekman-type damping $(-\Delta)^{-2\gamma}u$, with $\gamma\geq 0$. 
Written for the vorticity $\omega= \mathrm{curl}\,u := -\partial_yu_1+\partial_x u_2$, equations \eqref{eq:NSE} read
\begin{equation}\label{eq:vNSE}%\tag{\textsc{NSE}}
\begin{cases}
\partial_t \omega + u\cdot\nabla \omega  = \nu \Delta \omega - \alpha (-\Delta)^{-2\gamma}\omega+ \mathrm{curl}\,\partial_t W^\lambda, \\
 u = \nabla^\perp (-\Delta)^{-1} \omega := \begin{pmatrix} -\partial_y \\ \partial_x \end{pmatrix} \Delta^{-1} \omega. 
\end{cases}
\end{equation}
We will assume that the noise $W^\lambda(t)$ is given by
\[
	W^\lambda(t,x) = \sum_{j \in \Naturals} g^\lambda_j(x) W_j(t) 
\]
where $\{g^\lambda_j(x)\}$ is a sequence of smooth, mean zero divergence free vector fields on $\T_\lambda^2$ and $\set{W_j(t)}$ are a family of independent one-dimensional
Wiener processes on a common, canonical filtered probability space denote here as $(\Omega,\mathcal{F}, (\mathcal{F}_t),\P)$.
Define respectively what will be the average energy input per unit time per unit area (we will assume these quantities are independent of $\lambda$ for simplicity): 
\begin{align}\label{eq:eps}
\eps := \frac12\sum_{j}\fint_{\T_
\lambda^2}|g^\lambda _j(x)|^2\dx < \infty, 
\end{align}
and the average enstrophy per unit time per unit area: 
\begin{align}\label{eq:eta}
\eta := \frac12\sum_j\fint_{T_{\lambda}^2}|\mathrm{curl}\,g^\lambda_j(x)|^2\dx <\infty.
\end{align}
Indeed, one of the advantages of the white-in-time forcing is that the average energy and enstrophy input per unit time are independent of the solution. 
% For each $k\in \Z_{\lambda}^2 = \Z^2/\lambda$, let $\hat{g}^\lambda_j(k)$ be the Fourier transform of $g^\lambda_j(x)$ and define the sequence of measures $\{\hat{K}_\lambda(\dee\xi)\}_{\lambda >0}$ on $\R^2$
% \[
% 	\hat{K}_\lambda = \sum_{k\in Z_{\lambda}^2}\sum_{j}|\hat{g}_j^\lambda(k)|^2\delta_k,
% \]
% where $\delta_k$ denotes the delta measure on $\R^2$ centered at point $k \in \Z_\lambda^2 \subset \R^2$. Note that by definition
% \[
% 	\hat{K}_\lambda(\R^2) = \ep.
% \]

In this work we will concentrate on statistically stationary solutions (see Section \ref{sec:def} for rigorous definitions), which are expected to be the easiest setting in which to study turbulence. 
Let $u$ be a (statistically) stationary solution to the Navier-Stokes equations. Using \eqref{eq:NSE}, a simple application of It\^o's formula together with stationarity implies the balance
\begin{align}\label{eq:ENERbal}
 \nu\EE \norm{\grad u}^2_\lambda +   \alpha \EE \norm{(-\Delta)^{-\gamma}u}^2_\lambda =  \eps, 
\end{align}
where we are denoting $\norm{f}_\lambda := \left(\fint_{\mathbb T_\lambda^2} \abs{f(x)}^2 \dx\right)^{1/2}$. 
In the same way, from \eqref{eq:vNSE} it follows that
\begin{align}\label{eq:ENSTRObal}
\nu\EE \norm{\grad \omega}^2_\lambda +   \alpha\EE \norm{(-\Delta)^{-\gamma}\omega}^2_\lambda =  \eta. 
\end{align}

\subsection{Universal 2d turbulence laws}

% Two-dimensional turbulence plays a similar building-block role for the dynamics of plasmas transverse to a strong background magnetic field in certain regimes \red{[cite?]}

Modern understanding of 2d turbulence began in the foundational works \cite{Fjortoft53,Kraichnan67,Leith1968,Batchelor1969}, which first identified the characteristic \emph{dual cascade} picture.
Classically, this requires the injection of turbulent fluctuations via a forcing term, in order to produce a statistically stationary
state. %The removal of these fluctuations happens in two ways: at scales bigger than the injection scale, it is due to friction, while
%at scales smaller than the injection scale, it is caused by viscosity.
In fully developed 2d turbulence, \emph{two} inertial ranges are expected. Energy is expected to transfer from the injection scale to larger scales where it is damped by friction (an \emph{inverse cascade}) whereas the enstrophy is expected to be transferred from the injection scale down to smaller scales where it is dissipated via viscosity (a \emph{direct cascade}).
See \cite{Bernard00,Eyink96,Kupiainen2010,XB18} and the surveys \cite{MontKraichnan1980,BE12} for more discussions.
Such dual cascades with two inertial ranges are now understood to generally occur in turbulent systems with more than one positive conservation law in the inviscid limit; see discussions of direct and inverse cascades in various plasma and nonlinear wave systems in e.g. \cite{Biskamp2003,Nazarenko11,ZLF12}. 

In 3d turbulence, energy is observed to undergo a direct cascade from larger to smaller scales. The constant flux of energy through the inertial range leads to the celebrated Kolmogorov 4/5 law derived in the original K41 works \cite{K41a,K41b,K41c}; see \cite{Frisch1995} for an in-depth discussion.    
In 2d turbulence, constant energy and enstrophy fluxes in the respective inertial ranges imply exact relations for the third-order structure functions of the velocity, apparently derived more or less simultaneously by Bernard, Lindborg, and Yakhot \cite{Bernard99, Lindborg99, Yakhot99}. See also the more recent and detailed discussions in \cite{CP17,XB18}.  
The inverse cascade is expected over a range of scales $\ell_I \ll \abs{h} \ll \ell_\alpha$ (injection scale and friction scale respectively) and over this range we expect: 
\begin{align}
\EE \left(\delta_{h} u \cdot \frac{h}{\abs{h}}\right)^3 \sim \frac{3}{2} \eps |h|, \label{eq:BLYinv}
\end{align} 
where $\delta_h u(x):= u(x+h)-u(x)$ is the increment by the vector $h\in \R^2$. The quantity appearing on the left-hand side above is referred to as the third-order longitudinal structure function, and it is related to the energy flux through scale $\ell$ (see \cite{Frisch1995} for more discussion on this).
The positivity of the right-hand side is a sign that the energy flux is from smaller to larger scales. 
In the range of scales associated with the direct cascade we expect the following to hold over a range of scales $\ell_\nu \ll \abs{h} \ll \ell_I$ (viscous scale and injection scale respectively):
\begin{align} 
\EE \left(\delta_{h} u \cdot \frac{h}{\abs{h}}\right)^3 \sim \frac{1}{8} \eta |h|^3, \label{eq:BLYdirect}
\end{align} 
which is indicative of a direct cascade of enstrophy (see e.g. \cite{CP17}). As predicted by Eyink \cite{Eyink96}, one also expects Yaglom's law \cite{Yaglom49} for the vorticity in the inertial range $\ell_\nu \ll \abs{h} \ll \ell_I$ (Yaglom originally derived this prediction for passive scalar turbulence), 
\begin{align}
\EE \left( \abs{\delta_h \omega}^2 \delta_h u \cdot \frac{h}{\abs{h}}\right) \sim -2 \eta \abs{h}. \label{eq:YagIntro}
\end{align}
Over the inertial ranges, Batchelor-Kraichnan theory also predicts the power spectra, i.e. the ensemble-averaged distribution of energy density in frequency: 
\begin{subequations} 
\begin{align}
  \abs{k} \EE \abs{\hat{u}(k)}^2 & \approx \eps^{2/3}\abs{k}^{-5/3}, \quad\ \quad \ell_\alpha^{-1} \ll \abs{k} \ll \ell_I^{-1}, \label{ineq:BKspec1} \\
  \abs{k} \EE \abs{\hat{u}(k)}^2 & \approx \eta^{2/3}\abs{k}^{-3}, \quad\qquad \ell_I^{-1} \ll \abs{k} \ll \ell_\nu^{-1}. \label{ineq:BKspec2}
\end{align}
\end{subequations} 
Note that the $-3$ prediction for the velocity in the direct cascade inertial range is formally equivalent to a $-1$ spectrum on the enstrophy. 
It seems that rigorously deriving statements on power spectra such as \eqref{ineq:BKspec1} or \eqref{ineq:BKspec2} are likely to be significantly more difficult than providing rigorous proofs of flux laws such as \eqref{eq:YagIntro}. 
\begin{remark}
Notice that \eqref{ineq:BKspec2} specifically implies that the total amount of enstrophy in the inertial range diverges logarithmically in $\nu$ as $\nu \to 0$.
\end{remark}

Despite the variety of challenges in making accurate measurements, there has been many experiments, observations, and numerical simulations to test the ideas of 2d turbulence theory. 
For experiments and observations, see the surveys~\cite{BE12,Kellay2002,CerbusPhD2015}. One approach is to compare to atmospheric data~\cite{Lindborg99,Charney1971}, the second approach, convenient for the laboratory setting, is gravity driven soap film channels~\cite{Couder1989,Gharib1989,Rutgers2001,Kellay2017,Cerbus2013,Rivera2014,Rutgers1998}.
A third laboratory approach has been electromagnetically driven flows of thin, stably-stratified layers~\cite{Cardoso1994,Paret1998,Paret1997,Hansen1998,Paret1999}. %, also suitable for the laboratory setting.
%In this experimental setup, a cell is filled with a lighter and a heavier NaCl solution. Permanent magnets are placed under the cell and an electric current is driven through the cell from one side to the other. The interaction of the current and the magnetic field produce forces that drive the flow~\cite{Paret1998}.
The direct enstrophy cascade has been observed in~\cite{Paret1999,Kellay2017,Rivera2014,Cardoso1994,Rutgers1998} and the inverse cascade was studied e.g. in~\cite{Sommeria1986,Paret1998,Kellay2017,Paret1997,Rutgers1998,Bruneau2005,Chen2006}. Various structure functions were reported in~\cite{Kellay2002,Vorobieff1999,Cerbus2013,RiveraPhD2001,Rivera2003,CerbusPhD2015,Tabeling2002,Bruneau2005}. Numerical simulations to reproduce energy and enstrophy cascades and structure functions were performed in~\cite{Lilly1969,Frisch1984,Lindborg2000,Boffetta2007,Boffetta2010,Bruneau2005,Chen2006,Xiao2009}. Broadly speaking, especially in the more recent experiments, the data is in agreement with the predictions of 2d turbulence theory but subtleties certainly remain -- see the discussions in e.g.~\cite{BE12} for a detailed account.
It is worth pointing out that Gaussian white-in-time stochastic forcing as in \eqref{eq:NSE}--\eqref{eq:vNSE} is standard in numerical experiments.
As is the practice of using hyperviscosity $\nu\Delta \mapsto -\nu (-\Delta)^{M}$ for some $M > 1$ and using ``hypofriction'' $-\alpha  \mapsto \alpha (-\Delta)^{-2\gamma}$.
This is often done to expand the size of the inertial range; it is well-understood that these changes do not significantly change inertial range statistics except close to $\ell_\nu$ and $\ell_\alpha$ where obviously the precise form of the dissipation plays an important role.  See the discussion in \cite{BE12} and the references therein.

\subsection{Stationary solutions and weak anomalous dissipation}
Despite the fundamental importance of statistical theories of turbulence in physics and engineering, few mathematically rigorous works put the ideas on firm theoretical foundations.
%To our knowledge, the primary exceptions to this is the resolution of Onsager's conjecture in 3d (see \cite{Isett18} and the references therein \red{[Should provide more citations]}) and the proof of Yaglom's law in (Batchelor regime) passive scalar turbulence \cite{BBPS18} \red{[not sure how to discuss this sort of thing, and whether to do so here or later -- other mathematical works that need discussion: the other conditional works in 2d and 3d e.g. Theo and Eyink, Alexey and Roman, Hairer/Romito et al; other convex integration works, especially those in 2d; Flandoli/Maslowski and Hairer/Mattingly ergodicity results, Kuksin-Shirikiyan (both the book in general for well-posedness theory etc...should we mention the Kuksin measures, apparently also called `flucuation-dissipation measures', or do we not want to rub in the fact that these measures are definitely not related turbulence and just assume readers will figure that out on their own? Are there any other works?]} \red{[How do we write this paragraph without coming off as douchey?]}
Statistical laws describing constant flux of (inviscid) conserved quantities such as \eqref{eq:BLYinv}, \eqref{eq:BLYdirect}, and \eqref{eq:YagIntro}, are expected to be the easiest laws to verify mathematically rigorously.
Several works have previously appeared focusing on finding sufficient conditions to deduce variants of these laws in deterministic settings, mostly for 3d Navier-Stokes and Euler; see e.g. \cite{Nie1999,Duchon2000,Eyink2003,Drivas18}. Several works have studied sufficient conditions to obtain some estimates on the power spectrum in 3d, for example in the deterministic case \cite{CSint14} and the stochastic case \cite{FlandoliEtAl08}. Other works have been seeking a priori estimates in the high Reynolds number limit to provide some constraints on the possibilities; see e.g. \cite{CR,DFJ10,CTV14,BJMT14} and the references therein. The work on Onsager's conjecture in 3d can be seen as another kind of consistency check (see e.g. \cite{Eyink94,CET,Isett16} and the references therein).

In our previous work \cite{BCZPSW18}, we derived the weakest known sufficient condition to deduce the Kolmogorov 4/5 law for statistically stationary martingale solutions of the 3d Navier-Stokes equations. The condition is simply $\lim_{\nu \to 0}\nu \EE\norm{u}_{L^2}^2 = 0$, which we refer to as \emph{weak anomalous dissipation}.
As remarked in \cite{BCZPSW18}, this is equivalent to the assertion that the Taylor microscale decreases to zero as $\nu \to 0$. 
An analogous condition has been proven for passive scalars advected by a weakly mixing flow  \cite{BCZGH}, and
 was used to provide a complete proof of Yaglom's law \cite{Yaglom49} for (Batchelor-regime) passive scalar turbulence in \cite{BBPS18} (essentially \eqref{eq:YagIntro} but with $\omega$ replaced by a passively advected scalar in a vanishing diffusivity limit).
The work of \cite{BBPS18} appears to be the first proof of any scaling law from the classical statistical theory of turbulence. 
That such a weak condition is sufficient was crucial in \cite{BBPS18}; it seems that even for the much simpler case of (Batchelor-regime) passive scalar turbulence, getting a more quantitative understanding of the direct cascade is significantly more difficult. 

Relatively few works have considered scaling laws in the 2d case as it is notably more subtle than the 3d case. The exception we are aware of is \cite{Drivas18}, where a Lagrangian analogue of \eqref{eq:BLYinv} is derived in the deterministic case using relatively strong assumptions (but which are supported by experiments).
Perhaps the work closest to ours in spirit is that of \cite{XB18}; though not phrased in a mathematically rigorous manner, the arguments therein could be made rigorous by taking sufficiently strong hypotheses.  

In this work, we want to find the weakest possible conditions in order to provide mathematically rigorous statements of \eqref{eq:BLYinv}, \eqref{eq:BLYdirect}, and \eqref{eq:YagIntro}, in the hopes that they will eventuallly help lead to a complete proof (as \cite{BCZPSW18} did for \eqref{eq:YagIntro} for (Batchelor-regime) passive scalar turbulence \cite{BBPS18}). 
Motivated by our previous works \cite{BCZPSW18,BBPS18}, and discussions of 2d turbulence in \cite{Eyink96,Kupiainen2010}, in order to deduce a dual cascade (i.e. with both inertial ranges) in the large box limit, we use the following assumption. 
\begin{definition}[Weak anomalous dissipation] \label{def:WAD}
We say that a sequence $\set{u}_{\nu,\alpha > 0}$ of stationary solutions to \eqref{eq:NSE}-\eqref{eq:vNSE} satisfies \emph{weak anomalous dissipation} if
\begin{subequations}\label{eq:WAD}
\begin{align}
&\lim_{\nu \to 0}\sup_{\alpha \in (0,1)}  \nu\EE \norm{\omega}^2_\lambda = 0, \label{eq:WADa}\\
&\lim_{\alpha \to 0}\sup_{\nu \in (0,1)}  \alpha\EE \norm{(-\Delta)^{-\gamma}\omega}_\lambda^2 = 0. \label{eq:WADb}
\end{align}
\end{subequations}
\end{definition}

\begin{remark}
If one chooses the parameters $\nu,\alpha$ to be linked somehow, e.g. $\nu \approx \alpha$, then naturally the inner suprema can be dropped; see Remark \ref{rmk:nualphlink}.  
\end{remark}

% \begin{remark}
% From \eqref{
% \begin{align}
% \sup_{\nu \in (0,1)} \alpha \EE \norm{(-\Delta)^{-\gamma} \omega}^2_\lambda & \leq \eta \\
% \sup_{\alpha \in (0,1)}  \nu \EE \norm{\omega}_{\lambda}^2 & \leq \eps. 
% \end{align}
% However, we are unaware of any other a priori estimates aside from those that follow from these two assumptions combined with parabolic regularity theory. 
% \end{remark}

Equivalently, the above conditions can be stated in terms of the energy/enstrophy balance, which makes Definition \ref{def:WAD} look a little more like classical anomalous dissipation assumptions, in contrast to that used in \cite{BCZPSW18}. 

\begin{proposition}
Let $\set{u}_{\nu,\alpha > 0}$ be a sequence of stationary solutions to \eqref{eq:NSE}-\eqref{eq:vNSE}. Then \eqref{eq:WAD} holds if and only if
\begin{subequations}\label{eq:WAD2}
	\begin{align}
	& \lim_{\nu \to 0} \sup_{\alpha \in (0,1)}\abs{\eps -   \damp\EE \norm{(-\Delta)^{-\gamma}u}_\lambda^2} = 0,\label{eq:WAD2a}\\ 
	& \lim_{\alpha \to 0} \sup_{\nu \in (0,1)} \abs{\nu\EE \norm{\grad \omega}^2_\lambda - \eta} = 0. \label{eq:WAD2b}
	\end{align}
	\end{subequations}
\end{proposition} 
\begin{proof}
Rearranging the energy balance \eqref{eq:ENERbal} gives
	\begin{align}
	\abs{\eps - \damp\EE \norm{(-\Delta)^{-\gamma}u}^2_\lambda} =  \nu\EE \norm{\grad u}^2_\lambda. 
	\end{align}
Since $ \norm{\grad u}_\lambda$ is comparable to $\norm{\omega}_\lambda$, the equivalence between \eqref{eq:WADa} and \eqref{eq:WAD2a} 
is apparent. The second statement follows in the same way, using \eqref{eq:ENSTRObal}. 
\end{proof}
\begin{remark}
  The physical content of \eqref{eq:WAD2} is clear: as $\nu,\alpha \to 0$ \emph{all} of the energy is being dissipated by the large-scale damping and \emph{all} of the enstrophy is being dissipated by viscosity.
\end{remark}

In addition to conditions \eqref{eq:eps} and $\eqref{eq:eta}$ that the net energy and enstrophy input is finite we will also take the following, uniform in $\lambda$, regularity conditions on the noise:
\begin{assumption} \label{cond:g}
We will assume that $\eps$ and $\eta$ are \emph{independent} of $\lambda \geq 1$ and that
\begin{align}
\sup_{\lambda\in(1,\infty)}\sum_{j \in \Naturals} \norm{\grad^3 g^\lambda_j}_{\lambda}^2 & \lesssim 1,\\
\lim_{\delta \to 0} \sup_{\lambda \in (1,\infty)} \sum_{j \in \Naturals} \norm{(g^\lambda_j)_{\leq \delta}}_{\lambda}^2 & = 0,
\end{align}
where $f_{\leq \delta}$ denotes the restriction to frequencies less than $\delta$ (see Section \ref{sec:Note} for Fourier analysis conventions). 
\end{assumption}

\begin{remark}
The above conditions on the $\{g_j^\lambda\}$ ensure that the power spectrum of the noise is not too singular at low and high frequencies uniformly in $\lambda$. This will be important in showing that certain correlation functions of the noise converge appropriately in the $\nu\to 0$ and $\alpha \to 0$ limits. One can view this is as an assertion that the energy/enstrophy is mostly being injected at $\mathcal{O}(1)$ scales. 
\end{remark}

\subsection{Informal statements of main results}
As the full statements of the theorems can appear a little technical at first, we have opted to make abbreviated statements first that are more along the lines of the statements present in the physics literature. The full statements are made below in the respective sections. 

\subsubsection{The dual cascade}
The full statement of the dual cascade can be found in Sections \ref{sec:Direct} and \ref{sec:Inverse} below (for the direct and inverse cascades, respectively). 
\begin{theorem}[Informal characterization of the dual cascade] \label{thm:InformalDual}
Suppose that $\lambda = \lambda(\alpha) < \infty$ is a continuous monotone increasing function such that $\lim_{\alpha \to 0} \lambda = \infty$.
Let $\set{u}_{\nu,\alpha >0}$ be a sequence of statistically stationary solutions such that Definition \ref{def:WAD} holds. Then,
\begin{itemize}
\item[(i)] There exists a dissipative scale $\ell_\nu \in (0,1)$ satisfying $\lim_{\nu \to 0} \ell_\nu = 0$ such that the following laws hold over a \emph{small-scale} inertial range $(\ell_\nu,1)$ at asymptotically small scales:
\begin{align}
& \EE \fint_{\mathbb S} \fint_{\TT_\lambda^2} \abs{\delta_{\ell n} \omega}^2 \delta_{\ell n} u \cdot n \, \dee x \dee n \sim -2\eta \ell, \label{eq:yag} \\
& \EE \fint_{\mathbb S} \fint_{\TT_\lambda^2}  \abs{\delta_{\ell n} u}^2\left(\delta_{\ell n} u \cdot n\right) \dee x \dee n \sim \frac14\eta \ell^3, \label{eq:lind1}\\
& \EE \fint_{\mathbb S} \fint_{\TT_\lambda^2}   \left(\delta_{\ell n} u \cdot n\right)^3 \dee x \dee n \sim \frac{1}{8}\eta \ell^3. \label{eq:lind2}
\end{align}
\item[(ii)] There exists a damping scale $\ell_\alpha \in (1,\infty)$ satisfying $\lim_{\alpha \to 0} \ell_{\alpha} = \infty$ such that the following laws hold over a \emph{large-scale} inertial range $(1,\ell_\alpha)$ at asymptotically large scales:
\begin{align}
& \EE \fint_{\mathbb S} \fint_{\TT_\lambda^2}  \abs{\delta_{\ell n} u}^2\left(\delta_{\ell n} u \cdot n\right) \, \dee x \dee n \sim  2\eps \ell,\label{eq:inv1} \\
& \EE \fint_{\mathbb S} \fint_{\TT_\lambda^2}   \left(\delta_{\ell n} u \cdot n\right)^3 \dee x \dee n \sim \frac{3}{2}\eps \ell. \label{eq:inv2}
\end{align}
\end{itemize}
\end{theorem}

\begin{remark}
In our proofs we choose $\ell_\nu$ and $\ell_\alpha$ such that 
\begin{align}
\lim_{\nu \to 0}\frac{ \sup_{\alpha \in (0,1)} \nu \EE \norm{\omega}_{\lambda}^2}{\ell_{\nu}^2} & = 0, \\
\lim_{\alpha \to 0} \ell_\alpha^2 \left(\sup_{\nu \in (0,1)} \alpha \EE \norm{(-\Delta)^{-\gamma} \omega}_{\lambda}^2\right) & = 0. 
\end{align}
These choices should be interpreted as \emph{estimates} of the small scale inertial range $(\ell_\nu, 1)$ and large scale inertial range $(1,\ell_\alpha)$.  
The true inertial ranges could be larger.  
\end{remark}

\begin{remark}
Our proof also naturally provides error estimates in $\ell$ and $\nu,\alpha$. For example, for \eqref{eq:yag}, our proofs shows that there is an explicitly computable function $F_{g,\lambda}(\ell)$ which depends only on the noise such that for $\ell \ll 1$
\begin{align}
 \frac{1}{\ell}\EE \fint_{\mathbb S} \fint_{\TT_\lambda^2} \abs{\delta_{\ell n} \omega}^2 \delta_{\ell n} u \cdot n \, \dee x \dee n = F_{g,\lambda}(\ell) + \mathcal{O}\left(\frac{ \nu \EE \norm{\omega}_{\lambda}^2}{\ell^2} \right) + \mathcal{O}\left(\alpha \EE \norm{(-\Delta)^{-\gamma} \omega}_{\lambda}^2\right)
\end{align}
and that $F(\ell) = 2\eta + \mathcal{O}(\ell)$. In principle, one can provide a further expansions for $F$ as $\ell \to 0$; similar expansions are already required for several of our results (see also analogous calculations in \cite{XB18}). 
\end{remark} 

\begin{remark}
Provided we work with spatially homogeneous\footnote{A solution is spatially homogeneous if $u(\cdot,\cdot)$ has the same law as $u(\cdot,\cdot + y)$ for all $y \in \Real^2$. See e.g. \cite{Frisch1995,BCZPSW18,Basson2008} for more discussions statistical symmetries.} solutions, it seems that we could pass to the limit $L \to \infty$ \emph{first} while fixing all the other parameters and study stationary, homogeneous solutions to damped 2d Navier-Stokes on $\Real^2$ (see e.g. \cite{Basson2008,VF12} for more details on how this could be done).  In this case one could also impose statistical isotropy (see the discussions in \cite{DFK06}). 
This has some appeal to it, but it is also mathematically more technical and is not necessary to isolate the inverse cascade here, as one is in any case necessarily constrained to a finite inertial range for all non-zero $\alpha$. 
\end{remark}

\begin{remark}
As in \cite{BCZPSW18}, if one has the suitable statistical symmetries, the averages in $x$ and/or $n$ can be removed. 
\end{remark}

\begin{remark} \label{rmk:nualphlink}
In our results, the order in which $\nu$ and $\alpha$ are taken to 0 does not matter (see Theorem \ref{thm:direct} and \ref{thm:inverse} for precise statements).
This is a consequence of the strength of the anomalous dissipation assumption in Definition \ref{def:WAD}.
In particular, we could also consider the case in which $\nu$ and $\alpha$ are related, in which case  \eqref{eq:WAD} could be relaxed significantly.
For example if one takes $\nu \approx \alpha$ (that is, comparable up to multiplicative constants), then we replace \eqref{eq:WAD} with 
\begin{align}
&\lim_{\nu \to 0}  \nu\EE \norm{\omega}^2_\lambda = 0, \\ 
&\lim_{\alpha \to 0} \alpha\EE \norm{(-\Delta)^{-\gamma}\omega}_\lambda^2 = 0. 
\end{align}
This case could potentially be the most amenable to rigorous mathematical analysis. 
\end{remark}

\subsubsection{Necessary conditions}{}
It is of course natural to ask how close the above conditions are to being \emph{necessary}.
This can be formulated rigorously via pre-compactness or equi-integrability for the dissipation and damping (for the direct and inverse cascades respectively) which rule out any kind of damping/dissipation anomalies. 

\begin{theorem}[Necessary conditions for the dual cascade] \label{thm:Nec}
Fix $\gamma \geq 0$ and let $\{u\}_{\nu,\alpha >0}$ be a sequence of statistically stationary solutions.
\begin{itemize}
\item Suppose that the following precompactness conditions hold:
\begin{subequations}
\begin{align}
&\lim_{|h|\to 0}\sup_{\nu,\alpha \in(0,1)}\nu \E\|\nabla \delta_h\omega\|_\lambda^2 = 0 \label{def:PCdiss} \\ 
&\lim_{|h|\to 0}\sup_{\nu,\alpha \in(0,1)} \alpha\E\|(-\Delta)^{-\gamma}\delta_h \omega\|_{\lambda}^2 = 0. \label{def:PCdamp}
\end{align}
\end{subequations}
Then the scaling laws \eqref{eq:yag}, \eqref{eq:lind1} and \eqref{eq:lind2} for the direct cascade cannot hold. Specifically,
\begin{subequations}
	\begin{align}
	&\lim_{\ell\to 0}\sup_{\nu,\alpha \in (0,1)}\frac{1}{\ell}\left|\E\fint_{\S}\fint_{\TT_\lambda^2} \abs{\delta_{\ell n} \omega}^2 \delta_{\ell n} u \cdot n \, \dee x \dee n\right| = 0,\label{eq:necessary1}\\
	&\lim_{\ell\to 0}\sup_{\nu,\alpha \in (0,1)}\frac{1}{\ell^3}\left|\E\fint_{\S}\fint_{\TT_\lambda^2} \abs{\delta_{\ell n} u\cdot n}^2 \delta_{\ell n} u \cdot n \, \dee x \dee n\right| = 0,\label{eq:necessary2}\\
	&\lim_{\ell\to 0}\sup_{\nu,\alpha \in (0,1)}\frac{1}{\ell^3}\left|\E\fint_{\S}\fint_{\TT_\lambda^2} (\delta_{\ell n} u\cdot n)^3 \dee x \dee n\right| = 0.\label{eq:necessary3}
\end{align}
\end{subequations}
\item Suppose that the following equi-integrability conditions at low frequencies hold:
\begin{subequations}
\begin{align}
&\lim_{\delta \to 0}\sup_{\nu,\alpha \in (0,1)}\nu\E\|\nabla u_{\leq \delta}\|_{\lambda}^2 =0, \label{def:EQdiss} \\
&\lim_{\delta\to 0} \sup_{\nu,\alpha \in (0,1)} \alpha \E\|(-\Delta)^{-\gamma}u_{\leq \delta}\|_{\lambda}^2 = 0, \label{def:EQdamp}
\end{align}
\end{subequations}
then the scaling laws \eqref{eq:inv1} and \eqref{eq:inv2} for the inverse cascade cannot hold. Specifically
\begin{subequations}
	\begin{align}
	&\lim_{\ell\to \infty}\sup_{\nu,\alpha \in (0,1)}\frac{1}{\ell}\left|\E\fint_{\S}\fint_{\TT_\lambda^2} \abs{\delta_{\ell n} u\cdot n}^2 \delta_{\ell n} u \cdot n \, \dee x \dee n\right| = 0,\label{eq:necessary4}\\
	&\lim_{\ell\to \infty}\sup_{\nu,\alpha \in (0,1)}\frac{1}{\ell}\left|\E\fint_{\S}\fint_{\TT_\lambda^2} (\delta_{\ell n} u\cdot n)^3 \dee x \dee n\right| = 0.\label{eq:necessary5}
\end{align}
\end{subequations}
\end{itemize}
\end{theorem}

\begin{remark}
By writing finite differences on the Fourier side, the pre-compactness condition \eqref{def:PCdiss} is equivalent to
\begin{align*}
\lim_{|h|\to 0}\sup_{\nu,\alpha \in(0,1)}\nu \E\|\nabla \delta_h\omega\|_\lambda^2 = 0 \Leftrightarrow  \lim_{N \to \infty}\sup_{\nu,\alpha \in(0,1)}\nu \E\|\nabla \omega_{\geq N}\|_\lambda^2 = 0, 
\end{align*}
where $f_{\geq N}$ denotes Fourier projection to frequencies larger than $N$. 
\end{remark}

\begin{remark}
Condition \eqref{def:PCdiss} is really the pre-compactness condition that rules out a dissipation anomaly. Condition \eqref{def:PCdamp} on the other hand seems purely technical and is used to rule out that the large-scale damping is playing a major role in the small-scale inertial range. Note that for $\nu \approx \alpha$, \eqref{def:PCdiss} is strictly much stronger than \eqref{def:PCdamp}, and so only \eqref{def:PCdiss} is needed in that case. Analogously, it is the equi-integrability \eqref{def:EQdamp} that rules out a large-scale damping anomaly, whereas \eqref{def:EQdiss} rules out that the viscosity plays a major role in the large-scale inertial range (and for $\nu \approx \alpha$, only \eqref{def:EQdamp} is needed).  
\end{remark}

% \begin{remark}
% It is not hard to check that Theorem \ref{thm:InformalDual} (and Theorems \ref{thm:NonUniIsoDirect} and \ref{thm:NonUniIsoInv} below) also extend to the hyperviscosity case $\nu \Delta \omega \mapsto - \nu (-\Delta)^{M} \omega$ with $M > 1$ (the definition for 
% \end{remark}

\subsubsection{Isolated cascades} \label{sec:IsoCasInt}

The dual cascade stated in Theorem~\ref{thm:InformalDual} is significantly more complicated than the corresponding theorem for the 3d case \cite{BCZPSW18}, which only has one cascade.
A very natural question is whether one can set up the problem in order to see only one of the cascades, e.g. only the direct cascade or only the inverse cascade.
Indeed, most experiments and computer simulations have only captured one of the cascades. For example, the relatively recent \cite{Boffetta2010} was the first numerical simulation to give convincing evidence for both cascades in the same simulation. See \cite{BE12} for more discussions. 

First, we consider the problem of fixing $\alpha$ and $\lambda$ and sending $\nu \to 0$ and isolating only a small-scale inertial range with a direct cascade of enstrophy.
Several things are different here: first, note that condition \eqref{eq:WADa} on the vorticity is automatic and in fact is satisfied in a quantitative sense: for $\theta = \frac{2\gamma}{2\gamma + 1}$,   
\begin{align}
\nu \EE \norm{\omega}_{L^2}^2 \leq  \nu^{1-\theta} \left( \EE \norm{(-\Delta)^{-\gamma}\omega}_{\lambda}^{2}   \right)^{1-\theta} \left(\EE \nu \norm{\grad \omega}_{\lambda}^2 \right)^{\theta} \lesssim \nu^{1-\theta}. 
\end{align}
On the other hand, since we are not taking $\alpha \to 0$, the effect of $-\alpha(-\Delta)^{-2\gamma}\omega$ on the global enstrophy budget will not vanish.
%Hence, in order to see a non-trivial direct cascade, we need to replace condition \eqref{eq:WADa} that ensures that not all of the enstrophy is dissipated by the damping: 
Subject to essentially the same mild precompactness condition  as above in \eqref{def:PCdamp}, we show that the third order structure functions are non-trivial \emph{if and only if} not all of the enstrophy is dissipated by the damping, namely
\begin{align}
\liminf_{\nu \to 0} \nu \EE \norm{\grad \omega}_{\lambda}^2 > 0. \label{eq:CAD}
\end{align}
This is, indeed,  the classical characterization of anomalous dissipation.
Given that the large scale damping leads to a $\nu$-independent upper bound in $L^2(\Omega;H^{-2\gamma})$, this condition implies a flux of enstrophy from large scales to small scales.
It was shown in \cite{CR} that \eqref{eq:CAD} fails if $\gamma = 0$, that is, the damping ends up dissipating all of the enstrophy. 
Hence, in order to isolate a direct cascade, we will also assume $\gamma > 0$, so that the damping is a higher order effect at small scales. 

%Unfortunately we are not currently aware of a way to prove the third order scaling laws only using \eqref{eq:CAD}; we currently require a very minimal amount of $\nu$-independent  regularity, specifically, similar to Theorem \ref{thm:Nec}, we need that $\set{\omega}_{\nu \in (0,1)}$ is pre-compact in $H^{-2\gamma}$ (condition \eqref{def:HnegPreCom} below).
%This ensures that the large-scale damping is only having an indirect effect by adjusting the value in \eqref{eq:CAD} (see \eqref{def:etanu} below).  

    %However, under this pre-compactness we can show that \eqref{eq:CAD} is \emph{necessary and sufficient} for a direct cascade. 

We refer to the resulting cascade as ``non-uniform'' as it does not (a priori) hold uniformly with respect to the parameters $\alpha,\lambda$.
As above, we state the result somewhat informally. The precise formulation is analogous to that of Theorem \ref{thm:direct} below. 
We remark that condition \eqref{def:HnegPreCom} seems mild (see Remark \ref{rmk:BKregIsoDir}) but a priori, the $L^2(\Omega;H^{-2\gamma})$ norm of the vorticity is only uniformly bounded (from \eqref{eq:ENSTRObal}), not pre-compact.

\begin{theorem}[Isolated (non-uniform) direct cascade] \label{thm:NonUniIsoDirect}
Let $\gamma > 0$ and suppose $\alpha, \lambda$ are fixed.
Suppose that $\set{u}_{\nu > 0}$ is a sequence of statistically stationary solutions such that the following pre-compactness in $H^{-2\gamma}$ holds:       %$\set{(-\Delta)^{-\gamma}\omega}_{\nu \in (0,1)}$ is pre-compact in $L^2$ .
\begin{align}
\lim_{\abs{h} \to 0} \sup_{\nu \in (0,1)} \EE \norm{(-\Delta)^{-\gamma}\delta_h\omega}_{\lambda}^2 = 0. \label{def:HnegPreCom}
\end{align}
%\end{itemize} 
Define
	\begin{align}
	\eta^\ast_\nu  = \nu \EE \norm{\grad \omega}_{\lambda}^2 = \eta -  \alpha \EE \norm{(-\Delta)^{-\gamma}\omega}_\lambda^2. \label{def:etanu}
    \end{align}
If $\liminf_{\nu \to 0} \eta^\ast_\nu > 0$ then there exists $\ell_\nu \in (0,1)$ satisfying $\lim_{\nu \to 0} \ell_\nu = 0$ such that the following laws hold over a \emph{small-scale} inertial range $(\ell_\nu,1)$ at asymptotically small scales:
\begin{align}
& \EE \fint_{\mathbb S} \fint_{\TT_\lambda^2} \abs{\delta_{\ell n} \omega}^2 \delta_{\ell n} u \cdot n \dd x \dd n \sim -2\eta^\ast_\nu \ell, \label{eq:YagIso} \\
& \EE \fint_{\mathbb S} \fint_{\TT_\lambda^2}  \abs{\delta_{\ell n} u}^2\left(\delta_{\ell n} u \cdot n\right) \dd x \dd n \sim \frac{1}{4}\eta^\ast_\nu \ell^3, \label{eq:Lind1Iso} \\
& \EE \fint_{\mathbb S} \fint_{\TT_\lambda^2}   \left(\delta_{\ell n} u \cdot n\right)^3 \dd x \dd n \sim \frac{1}{8}\eta^\ast_\nu \ell^3. \label{eq:Lind2Iso} \\
\end{align}
On the other hand, if $\lim_{\nu \to 0} \eta^\ast_\nu = 0$, then no non-trivial third order scaling law holds (in the same sense as Theorem \ref{thm:Nec}). 
In particular, if the pre-compactness \eqref{def:HnegPreCom} holds, then non-trivial third order scaling laws hold \textbf{if and only if} $\liminf_{\nu \to 0} \eta^\ast_\nu > 0$.  
\end{theorem}
\begin{remark} \label{rmk:ADstrongHs}
By interpolation against \eqref{eq:ENERbal}, note that under the setting of Theorem \ref{thm:NonUniIsoDirect}, \eqref{eq:CAD} implies the following blow-ups of all higher Sobolev norms:
for all $s > 1$,
\begin{align}
   \lim_{\nu \to 0} \nu \EE \norm{\omega}_{H^s}^2 = \infty. 
\end{align}
\end{remark}

\begin{remark}
One expects that the above requires $\alpha$ to be chosen small relative to e.g. $\eta$, in order to induce the hydrodynamic instabilities necessary to start a cascade. 
\end{remark}
\begin{remark} \label{rmk:BKregIsoDir}
The direct cascade power spectrum \eqref{ineq:BKspec2} predicts that $\omega$ should be uniformly bounded in $L^2(\Omega;H^{-2\gamma'})$ for all $\gamma' > 0$. As $H^{-2\gamma'}$ compactly embeds in $H^{-2\gamma}$ for all $0 < \gamma' < \gamma$, \eqref{ineq:BKspec2} would imply pre-compactness \eqref{def:HnegPreCom}.
Experiments suggest \eqref{ineq:BKspec2} is reasonably accurate (though perhaps not exact), so it seems quite reasonable to expect \eqref{def:HnegPreCom}, at least for $\gamma$ sufficiently far from zero. Providing a complete mathematical proof, however, might be challenging. 
\end{remark}

\begin{remark}
	Note that even without taking a large box limit, we still have 
	\begin{align}
	\lim_{\nu \to 0} \alpha \EE \norm{(-\Delta)^{-\gamma}u}_\lambda^2 = \eps, 
	\end{align}
	so that all of the energy is ultimately dissipated by large-scale damping. 
\end{remark}

Next, we turn to an isolated \emph{inverse} cascade.
The conditions we require as $\alpha \to 0$, $\lambda \to \infty$ are quite analogous: that the effect of damping does not vanish combined with equi-integrability at low frequencies (the analogue of the pre-compactness used above in \eqref{def:HnegPreCom}).
Below, $f_{\leq \delta}$ denotes projection to Fourier frequencies $\leq \delta$; see Section \ref{sec:Note}. 
As above, we state the result somewhat informally. The precise formulation is analogous to that of Theorem \ref{thm:inverse} below.

\begin{theorem}[Isolated (non-uniform) inverse cascade] \label{thm:NonUniIsoInv}
Let $\gamma \geq 0$ and $\nu$ fixed. Suppose that $\lambda = \lambda(\alpha) < \infty$ is a continuous monotone increasing function such that $\lim_{\alpha \to 0} \lambda = \infty$.
Suppose that $\set{u}_{\alpha}$ is a sequence of statistically stationary solutions such that the following equi-integrability of enstrophy at low frequencies holds 
\begin{align}
\lim_{\delta \to 0}\limsup_{\lambda \to \infty}\EE \norm{\omega_{\leq \delta}}_{\lambda}^2 = 0. \label{ineq:lowfreqcomp} 
\end{align}
% \begin{itemize}
% \item the following anomalous damping holds
% 	\begin{align}
% 	\label{eq:kupianen1}
% \liminf_{\alpha \to 0} \alpha \EE \norm{(-\Delta)^{-\gamma} u}_{\lambda}^2 > 0; 
% \end{align}
% \item 
% \end{itemize}
Define
	\begin{align}
	\eps^\ast_\alpha  = \damp \EE\norm{(-\Delta)^{\gamma} u}^2_\lambda =  \eps -  \nu \EE \norm{\grad u}^2_\lambda. 
	\end{align}
If $\liminf_{\alpha \to 0} \eps^\ast_\alpha > 0$, then there exists $\ell_\alpha \in (1,\infty)$ satisfying $\lim_{\alpha \to \infty} \ell_{\alpha} = \infty$ such that the following laws hold over a \emph{large-scale} inertial range $(1,\ell_\alpha)$ at asymptotically large scales:
\begin{align}
& \EE \fint_{\mathbb S} \fint_{\TT_\lambda^2}  \abs{\delta_{\ell n} u}^2\left(\delta_{\ell n} u \cdot n\right)\dd x \dd n \sim  2\eps^\ast_\alpha \ell, \label{eq:34invCas1Iso}  \\
& \EE \fint_{\mathbb S} \fint_{\TT_\lambda^2}   \left(\delta_{\ell n} u \cdot n\right)^3 \dd x \dd n \sim \frac{3}{2}\eps^\ast_\alpha \ell. \label{eq:45invCas1Iso}
\end{align}
If $\lim_{\alpha \to 0} \eps^\ast_\alpha = 0$, then no non-trivial third order scaling law holds (in the same sense as Theorem \ref{thm:Nec}). 
In particular, if the equi-integrability \eqref{ineq:lowfreqcomp} holds, then non-trivial third order scaling laws hold \textbf{if and only if} $\liminf_{\alpha \to 0} \eps^\ast_\alpha > 0$.  
\end{theorem} 

\begin{remark}
Note that $\EE \norm{\grad u}^2_\lambda \lesssim 1$ uniformly in $\lambda$; condition \eqref{ineq:lowfreqcomp} should be thought of as a low frequency analogue of the pre-compactness assumed in Theorem \ref{thm:IsoDirect}. Finally, we remark that the inverse cascade spectrum in \eqref{ineq:BKspec1} formally predicts $\EE \norm{\omega_{\leq \delta}}_{\lambda}^2 \lesssim \delta^{1/3}$ (which is consistent with \eqref{ineq:lowfreqcomp}). 
\end{remark}

\begin{remark}
Under the conditions of Theorem \ref{thm:IsoDirect} we still have that all of the enstrophy is eventually being dissipated by viscosity
\begin{align}
\lim_{\alpha \to 0} \nu \EE \norm{\grad \omega}_{\lambda}^2 = \eta. 
\end{align}
\end{remark} 

By combining the ideas of Theorem \ref{thm:InformalDual} with Theorems \ref{thm:NonUniIsoDirect}, \ref{thm:NonUniIsoInv} one can also strengthen the conditions and obtain cascades that are uniform with respect to the other parameters. As these are essentially an easy adaptation of the proof of Theorem \ref{thm:InformalDual}, the proofs are omitted for the sake of brevity.

\begin{theorem}[Uniform isolated direct cascade] \label{thm:IsoDirect}
Let $\gamma > 0$ and suppose that $\set{u}_{\alpha,\nu > 0}$ is a sequence of statistically stationary solutions such that $\eqref{eq:WADa}$ holds and the following precompactness holds       
\[
\lim_{\abs{h} \to 0} \sup_{\nu \in (0,1)} \EE \norm{(-\Delta)^{-\gamma}\delta_h\omega}_{\lambda}^2 = 0.
\]
Then the scaling laws \eqref{eq:yag}, \eqref{eq:lind1}, and \eqref{eq:lind2} hold uniformly in $\alpha$.
\end{theorem}

\begin{theorem}[Uniform isolated inverse cascade]
Let $\gamma \geq 0$ and suppose that $\set{u}_{\alpha,\nu > 0}$ is a sequence of statistically stationary solutions such that $\eqref{eq:WADb}$ holds and the following equi-integrability condition holds at low frequencies
\begin{align}
\lim_{\delta \to 0}\limsup_{\lambda \to \infty}\EE \norm{\nabla u_{\leq \delta}}_{\lambda}^2 = 0.
\end{align}
then the scaling laws \eqref{eq:inv1} and \eqref{eq:inv2} hold uniformly in $\nu$.
\end{theorem}

\subsection{Notation and conventions} \label{sec:Note}
We write $f \lesssim g$ if there exists $C > 0$ such that $f \leq C g$ (and analogously $f \gtrsim g$). We write $f \approx g$ if there exists $C > 0$ such that $C^{-1}f \leq g \leq Cf$.
Furthermore, we use hats to denote vectors with unit length, that is, if $h\neq 0$ is some vector, then $\hat{h}=h/|h|$. We denote the averaged $L^2$-norm in space by $\norm{f}_\lambda := \left(\fint_{\mathbb T_\lambda^2} \abs{f(x)}^2 \dx\right)^{1/2}$
and will sometimes use the notation
\begin{equation}
L^p:=L^p(\T^2),\quad W^{s,p}:= W^{s,p}(\T^2), \quad H^s := W^{s,2}(\T^2), 
\end{equation}
to denote $L^p$ and $W^{s,p}$ spaces over $\T^2$.
%We will also sometimes abbreviate $L^p$, $W^{s,p}$ and $C^k$-spaces in the variables $(\omega,t,x,\ell)\in \Omega\times[0,T]\times\T^2\times\R_+$ in the following way:
%\begin{equation}
%L^p_x:=L^p(\T^2),\quad W^{s,p}_x:= W^{s,p}(\T^2),\quad L^q_t:=L^q([0,T]),\quad
% L^r_\omega:=L^r(\Omega),\quad C^k_\ell := C^k(\R_+),
%\end{equation}
%etc., or use combinations of these, for example, $L^q_tL^p_x$ denotes the space $L^q([0,T];L^p(\T^2))$ and $L^p_{t,x}$ denotes $L^p([0,T]\times\T^2)$.
%We denote by  $\Ldiv$ and $\Hdiv$ the completions of divergence-free smooth functions with zero average (denoted by $ C_\diverg^{\infty}$) with respect to the $L^2$ and $H^1$-norms. For any $\alpha \in \Real$, $\HH_x^\alpha$ will denote the usual homogeneous Sobolev spaces. 
%With a slight abuse of notation we will say a vector field $u(x)\in \R^2$ or a rank two tensor field $A\in \R^{2\times 2}$ belongs to a space $X$ if each component $u^i$ and $A_{ij}$ belongs to $X$. 

We will also make frequent use of component-free tensor notation. Specifically, given any two vectors $u$ and $v$ we will denote $u\tensor v$ the rank two tensor with components $(u\tensor v)_{ij} = u^iv^j$. Moreover given any two rank two tensors $A$ and $B$ we will denote $:$ the Frobenius product defined by, $A:B = \sum_{i,j} A_{ij}B_{ij}$ and the norm $|A| = \sqrt{A:A}$.

Finally, we use the following Fourier analysis conventions:
\begin{align}
\hat{f}(k) & = \fint_{\TT_\lambda^2} f(x) \e^{-ix \cdot k} \dx, 
\quad f(x)  = \frac{1}{\lambda^2}\sum_{k \in \mathbb Z^2_\lambda} \hat{f}(k) \e^{ix\cdot k}, \quad f_{\leq N}(x)  = \frac{1}{\lambda^2}\sum_{k \in \mathbb Z^2_\lambda : \abs{k} \leq N} \hat{f}(k) \e^{ix\cdot k},
\end{align}
where  $\Z_{\lambda}^2=\frac{2\pi}{\lambda}\Z^2$.

\section{Preliminaries and K\'arm\'an-Horvath-Monin relations}
\subsection{Statistically stationary mild solutions} \label{sec:def}

Unlike our previous work \cite{BCZPSW18} on 3d Navier-Stokes, the solutions of \eqref{eq:vNSE} are quite well-behaved for $\nu,\alpha > 0$ and $\lambda \in (1,\infty)$ and one does not have to work with weak solutions. Instead, we work with mild solutions, which are the stochastically strong solutions of the stochastic evolution equation.  

\begin{definition}
Given a complete filtered probability space $(\Omega,\mathcal{F},(\mathcal{F}_t)_{t\in[0,T]},\P)$, a mild solution $(\omega_t)$ to \eqref{eq:vNSE} is an $\mathcal{F}_t$ adapted process $\omega : [0,T]\times\Omega \to L^2$ satisfying
\[
	\omega_t = \e^{-\nu\Delta t}\omega_0 - \int_0^t \e^{-\nu\Delta(t-s)} (u_s \cdot \nabla \omega_s - \alpha(-\Delta)^{-2\gamma}\omega_s)\ds + \sum_{j\in \N}\int_0^t \e^{-\nu\Delta (t-s)}g^\lambda_j \dee W^j_s.
\]
\end{definition}

The follow well-posedness result is well-known (see for instance \cite{KS}).

\begin{proposition} \label{prop:WP}
Suppose $\ep$ and $\eta$ are finite and Assumption~\ref{cond:g}, then for all $\nu, \alpha >0$ and $\lambda \geq 1$, the system \eqref{eq:vNSE} admits a global-in-time, $\P$-a.s. unique, mild solution $(\omega_t)$ with initial data $\omega_0$. Moreover, $(\omega_t)$ defines a Feller Markov process which has at least one stationary probability measure $\mu$ supported on $H^3$. That is, a measure satisfying the following for all bounded, measurable $\phi: L^2 \to \Real$ and $t \geq 0$: 
\begin{align}
\int_{L^2} \EE\, \phi(\omega_t) \mu(\dd\omega_0) = \int_{L^2} \phi\, \dee \mu. 
\end{align}
\end{proposition}

\begin{remark}
Under various non-degeneracy conditions on the noise $W^\lambda_t$, one can prove that there is a unique stationary measure (see for instance \cites{FM95,HM06}). However, uniqueness of the stationary measure is irrelevant to our discussion.  %, which only requires that there is {\it some} noise  with $0 < \ep < \infty$.
\end{remark}

\begin{definition}
We call $(\omega_t)$ a \emph{statistically stationary} solution to \eqref{eq:vNSE} if for each $\tau >0$ the law of $(\omega_t)$ and $(\omega_{t+ \tau})$ are equal on $C(\R_+;L^2)$.
\end{definition}
A statistically stationary solution can be built from a stationary measure $\mu$ by starting the process $(\omega_t)$ with initial data $\omega_0$ distributed according to $\mu$. Consequently, at every later time $t>0$, the law of $\omega_t$ is also distributed like $\mu$. Therefore for each $\phi \in L^1(\mu)$, $T, t>0$,  a statistically stationary solution $(\omega_t)$ satisfies: 
\begin{align}
\frac{1}{T} \int_0^T \EE \,\phi(\omega_s) \ds  =  \EE\, \phi(\omega_t) = \int_{L^2} \phi\, \dee\mu. 
\end{align}

\subsection{K\'arm\'an-Horvath-Monin relations}

The fundamental energy balance identities for proving Theorem \ref{thm:InformalDual} are the K\'arm\'am-Horvath-Monin (KHM) relations for statistically stationary solutions.
They are a natural balance law between a two-point correlation function and its flux, a third order structure function. In this section, we collect the various KHM relations for the vorticity and velocity form of the 2d stochastic Navier-Stokes equations.
For 3D Navier-Stokes, the KHM relation was derived by K\'arm\'an and Howarth~\cite{deKarman1938}, and later generalized by Monin~\cite{MoninYaglom}.
In two dimensions, an analogous KHM relation was used by Eyink \cite{Eyink96} to predict that Yaglom's law \eqref{eq:YagIntro} holds (this is also how one proves this law for passive scalar turbulence; see \cite{BBPS18}).
For the velocity structure functions in two-dimensions, they were used in \cite{XB18}.

\subsubsection{Vorticity relations}
We define the two point correlators for the vorticity and curl of the noise,
\begin{align}
%\mathfrak{B}(y) & = \EE \fint_{\mathbb T^2_\lambda} \fint_{\mathbb S} \omega(t,x) \omega(t,x+rn) dx dS(n) \\
%\mathcal{E}(y) & = \EE \fint_{\mathbb T^2_\lambda} \fint_{\mathbb S} (-\Delta)^{-\gamma}\omega(t,x) (-\Delta)^{-\gamma}\omega(t,x+rn) dx dS(n) \\
\mathfrak{B}(y) & = \EE \fint_{\mathbb T^2_\lambda} \omega(x) \omega(x+y) \dee x \label{eq:Bdef}\\
\mathfrak{G}(y) &= \EE \fint_{\mathbb T^2_\lambda} (-\Delta)^{-\gamma}\omega(x) (-\Delta)^{-\gamma}\omega(x+y) \dx\label{eq:Gfrakdef}\\
\mathfrak{a}(y) & = \frac{1}{2}\sum_j\fint_{\T_\lambda^2} \curl g^\lambda_j(x) \curl g^\lambda_j(x+y)\dx\label{eq:afrakdef}
\end{align}
as well as the corresponding enstrophy flux structure function
\[
\mathfrak{D}(y) = \EE \fint_{\TT^2_\lambda} \abs{\delta_{y} \omega(x)}^2  \delta_y u(x)\, \dx. 
\]
Given Assumption~\ref{cond:g} and Proposition~\ref{prop:WP}, one can check that these quantities are all at least $C^3$.

The KHM relation for vorticity is then a relation between $\mathfrak{B},\mathfrak{G}, \mathfrak{a}$ and $\mathfrak{D}$ given by following Proposition.

\begin{proposition}[Vorticity KHM relation] \label{prop:vKHM}
Let $(\omega_t)$ be a statistically stationary solution to \eqref{eq:vNSE}. Then for the following relation holds
% Then, for any $\varphi = \varphi(y)$ a smooth, compactly supported test function on $\R^2$, there holds
\begin{align}
\grad\cdot \mathfrak{D}(y) = -4\nu \Delta \mathfrak{B}(y) + 4\alpha \mathfrak{G}(y)  - 4 \mathfrak{a}(y). \label{eq:vKHM}
\end{align}
% \begin{align}
% \frac{1}{2} \int_{\Real^2} \grad\varphi(y) \cdot \mathfrak{F}(y) \dy = 2\nu \int_{\Real^2} \Delta \varphi (y) \mathfrak{B}(y) \dy - 2\alpha \int_{\Real^2} (-\Delta)^{-2\gamma}\varphi (y) \mathfrak{B}(y) \dy  + 2 \int_{\Real^2} \varphi(y) \mathfrak{a}(y) \dy. \label{eq:vKHM}
% \end{align}
\end{proposition}

The proof of this relation is via a simplification of the argument used in \cite{BCZPSW18} for the 3d Navier-Stokes equations. We omit the proof due to its similarity with \cite{BCZPSW18}. %(see also \cite{BBPS18} for an equivalent proof in the case of passive scalars). 

Using the divergence theorem and integrating both sides of $\eqref{eq:vKHM}$ over $\{|y|\leq \ell\}$, we obtain a formula for the spherically averaged structure function
\[
\bar{\mathfrak{D}}(\ell) := \EE \fint_{\S}\fint_{\T_\lambda^2}  |\delta_{\ell n} \omega(x)|^2 \delta_{\ell n} u(x)\cdot n \, \dee x \dee n, 
\]
in terms of spherically averaged correlation functions,
\begin{align}
	\bar{\mathfrak{B}}(\ell) &= \fint_{\S}\mathfrak{C}(\ell n)\dee n,\label{eq:Bbardef}\\
	\bar{\mathfrak{G}}(\ell) &= \fint_{\S}\mathfrak{G}(\ell n)\dee n,\label{eq:Gfrakbardef}\\
	\bar{\mathfrak{a}}(\ell) &= \fint_{\S}\mathfrak{a}(\ell n)\dee n.\label{eq:afrakbardef}
\end{align}
% \begin{align}
% 	\bar{\mathfrak{C}}(\ell) &= \E\fint_{\S}\fint_{\mathbb T^2_\lambda} \omega(x) \omega(x+ \ell n) \dee x \dee n\\
% 	\bar{\mathfrak{G}}(\ell) &= \E\fint_{\S}\fint_{\mathbb T^2_\lambda} (-\Delta)^{-\gamma}\omega(x) (-\Delta)^{-\gamma}\omega(x+\ell n)\dee x \dee n\\
% 	\bar{\mathfrak{a}}(\ell) &= \frac{1}{2}\sum_j\fint_{\S}\fint_{\T_\lambda^2} \curl g^\lambda_j(x) \curl g^\lambda_j(x+\ell n)\dx\dee n
% \end{align}
which is stated as follows. 
\begin{lemma}\label{lem:Yag-KHM-Sphere} The following formula holds for each $\ell >0$
\begin{align}
\bar{\mathfrak{D}}(\ell) &=- 2\nu \ell \fint_{|y|\leq \ell} \Delta \mathfrak{B}(y)\dy + 2\damp \ell \fint_{|y|\leq \ell} \mathfrak{G}(y)\dy - 2 \ell\fint_{|y|\leq \ell} \mathfrak{a}(y)\dy\\
&= -4\nu \bar{\mathfrak{B}}^\prime(\ell) + \frac{4\damp}{\ell} \int_{0}^\ell r \bar{\mathfrak{G}}(r)\dr - \frac{4}{\ell}\int_{0}^\ell r \bar{\mathfrak{a}}(r)\dr.
\end{align}
\end{lemma}

% \subsection{K\'arm\'an-Horvath-Monin relations for velocity}

\subsubsection{Velocity relations}
Similarly, when dealing with the velocity form of the equation, we define the two point correlation tensors 

\begin{align}
\Gamma(y) &:= \E\fint_{\T^2_\lambda} u(x)\tensor u(x+y)\dx,\label{eq:Gamdef}\\
\da(y) &:= \E\fint_{\T^2_\lambda} (-\Delta)^{-\gamma}u(x)\tensor (-\Delta)^{-\gamma}u(x+y)\dx,\label{eq:Gdef}\\
a(y) &:= \frac{1}{2}\sum_j\fint_{\T^2_\lambda} g^\lambda_j(x)\tensor g^\lambda_j(x+y)\dx.\label{eq:adef}
\end{align}
as well as the flux structure function, defined for each $j= 1,2$ by
\[
	D^j(y) = \E\fint_{\T^2_\lambda} (\delta_y u(x)\tensor \delta_y u(x)) \delta_y u^{j}(x)\dx.
\]
Given Assumption~\ref{cond:g} and Proposition~\ref{prop:WP}, one can check that these quantities are all at least $C^4$.
The following KHM relation for the velocity was proved in \cite{BCZPSW18} in the 3d case. It is stated in the following radially symmetric weak form to avoid contributions from the pressure. The only difference between the 2d and 3d cases are the values of the constants.

\begin{proposition}[Velocity KHM relation]\label{prop:KHM-vel}
Let $(u_t)$ be a statistically stationary solution to \eqref{eq:NSE}, and let $\eta(y) = (\eta_{ij}(y))_{ij=1}^2$ be a smooth test function of the form
\[
	\eta(y) = \phi(|y|)I + \varphi(|y|)\hat{y}\tensor\hat{y}, \quad \hat{y} = \frac{y}{|y|},
\]
where $\phi(\ell)$ and $\varphi(\ell)$ are smooth and compactly supported on $(0,\infty)$. Then the following identity holds
\begin{equation}\label{eq:KHMstat}
	\sum_{j =1}^2 \int_{\R^2} \partial_{j}\eta(y):D^j(y)\dy = 4 \nu \int_{\R^2} \Delta \eta(y):\Gamma(y)\dy - 4\alpha \int_{\R^2}\eta(y):G(y)\dy + 4\int_{\R^2}\eta(y):a(y)\dy.
\end{equation}
\end{proposition}

Similar to Lemma \ref{lem:Yag-KHM-Sphere} we can use Proposition \ref{prop:KHM-vel} to deduce formulas for the spherically averaged energy flux structure function
\[
\bar{D}(\ell) := \EE  \fint_{\S} \fint_{\T^2_\lambda}|\delta_{\ell n} u(x)|^2 \delta_{\ell n} u(x)\cdot n \, \dee x \dee n,
\]
and the spherically averaged correlators
\begin{align}
\bar{\Gamma}(\ell) &:=  \fint_{\S}\tr\Gamma(\ell n)\dn,\label{eq:Gambardef}\\
\bar{\da}(y) &:=  \fint_{\S}\tr \da(\ell n)\dn,\label{eq:Gbardef}\\
\bar{a}(y) &:= \fint_{\S}\tr a(\ell n)\dn.\label{eq:abardef}
\end{align}
% \begin{align}
% \bar{\Gamma}(\ell) &:=  \E \fint_{\S}\fint_{\T^{2}_\lambda} u(x)\cdot u(x+\ell n)\dx\dn\\
% \bar{\da}(y) &:=  \E\fint_{\S}\fint_{\T^{2}_{\lambda}} (-\Delta)^{-\gamma} u(x)\cdot (-\Delta)^{-\gamma}u(x+\ell n)\dx\dn\\
% \bar{a}(y) &:= \frac{1}{2}\sum_j\E\fint_{\S}\fint_{\T^2_\lambda} g^\lambda(x)_j\cdot g^\lambda_j(x+\ell n)\dx\dn.
% \end{align}

\begin{lemma}\label{lem:S0-Sphere} The following identity holds for each $\ell >0$
\begin{align}
\bar{D}(\ell) &=- 2\nu \ell \fint_{|y|\leq \ell} \Delta \Tr \Gamma(y)\dy + 2\damp \ell \fint_{|y|\leq \ell} \tr\da(y)\dy - 2\ell\fint_{|y|\leq \ell} \tr a(y)\dy\\
&= -4\nu \bar{\Gamma}^\prime(\ell) + \frac{4\damp}{\ell} \int_{0}^\ell r \bar{G}(r)\dr - \frac{4}{\ell}\int_0^\ell r \bar{a}(r)\dr.
\end{align}
\end{lemma}

Finally, we can also write a formula for the so-called longitudinal structure function
\[
	\bar{D}_{||}(\ell) := \EE\fint_{\S} \fint_{\T^2_\lambda}  (\delta_{\ell n} u(x)\cdot n)^3 \, \dee x \dee n, 
\]
in terms of $\bar{D}(\ell)$ and longitudinal versions of the correlation functions
\begin{align}
\bar{\Gamma}_{||}(\ell) &:= \fint_{\S}(n\tensor n) :\Gamma(\ell n)\dn,\label{eq:Gamlongdef}\\
\bar{\da}_{||}(\ell) &:= \fint_{\S}(n\tensor n) : G(\ell n)\dn,\label{eq:Glongdef}\\
\bar{a}_{||}(\ell) &:= \fint_{\S}(n\tensor n): a(\ell n)\dn.\label{eq:alongdef}
\end{align}
% \begin{align}
% \bar{\Gamma}_{||}(\ell) &:= \E \fint_{\S}\fint_{\T^2_\lambda} (u(x)\cdot n )(u(x+\ell n)\cdot n)\dx\dn\\
% \bar{G}_{||}(\ell) &:= \E\fint_{\S}\fint_{\T^2_\lambda} ((-\Delta)^{-\gamma}u(x)\cdot n )((-\Delta)^{-\gamma}u(x+\ell n)\cdot n)\dx\dn\\
% \bar{a}_{||}(\ell) &:= \frac{1}{2}\sum_{j}\E\fint_{\S}\fint_{\T^2_\lambda} (g^\lambda_j(x)\cdot n )(g^{\lambda}_j(x+\ell n)\cdot n)\dx\dn
% \end{align}

\begin{lemma}\label{lem:khmparallel}
The following identity holds for each $\ell >0$
\begin{align}
	\bar{D}_{||}(\ell) &= - 4\nu \overline\Gamma_{||}^\prime(\ell) +\frac{2}{\ell^3}\int_0^\ell r^2 \bar{D}(r)\dr + \frac{2\alpha}{\ell} \fint_{|y|\leq \ell} h\tensor h: \da(y)\dee y - \frac{2}{\ell}\fint_{|y|\leq \ell} y\tensor y :  a(y)\dee y \\
	&= - 4\nu \overline\Gamma_{||}^\prime(\ell) +\frac{2}{\ell^3}\int_0^\ell r^2 \bar{D}(r)\dr + \frac{4\damp}{\ell^3} \int_0^\ell r^3 \bar{\da}_{||}(r)\dr - \frac{4}{\ell^3}\int_0^\ell r^3 \bar{a}_{||}(r)\dr.
\end{align}
% where
% \[
% 	\bar{\Gamma}_{||}'(\ell) = \E\fint_{\S}\left(\fint_{\T_{\lambda}^2}(n\cdot u)(n\tensor n:\nabla u(x + n\ell ))\,\dx\right)\dee S(n).
% \]
\end{lemma}

\subsection{Large-scale cancellation lemmas}
The following lemma provides weak conditions necessary to prove that certain integrals of two-point correlation functions vanish at large scales.
This can be viewed as a cancellation lemma.
\begin{lemma}[Large scale cancellations] \label{lem:LrgCan1}
Let $\set{f^{\lambda}}_{\lambda \geq 1}$ be a sequence of random scalars (depending potentially also on $\alpha$ and $\nu$) and define
\begin{align}
F^\lambda(y) = \EE \fint_{\TT^2_\lambda}f^\lambda(x+y) f^\lambda(x) \dx. 
\end{align}
Suppose that the following two conditions hold (uniformly in $\nu$):
\begin{subequations}
\begin{align}
\lim_{\delta \to 0} \sup_{\lambda \geq 1} \EE \norm{f^\lambda_{\leq \delta}}_{\lambda}^2 & = 0, \label{c:EquiInt} \\
\sup_{\lambda \geq 1} \EE\norm{f^\lambda}_\lambda^2 & < \infty.  \label{c:apBd}
\end{align}
\end{subequations}
Then, 
\begin{align}
\lim_{\ell_I\to \infty}\sup_{\lambda} \sup_{\ell \in (\ell_I, \frac{1}{2}\lambda)} \abs{ \fint_{\{|y|\leq \ell\}} F^\lambda (y) \dy} = 0. 
\end{align}
\end{lemma} 
\begin{proof}
By Fourier analysis and Fubini, 
\begin{align}
\frac{1}{\ell^2} \int_{\{|y|\leq \ell\}} F^\lambda (y) \dy = \sum_{k \in \mathbb Z^2_\lambda} \Phi_\ell(k) \E|\hat{f}^\lambda(k)|^2, 
\end{align}
where 
\begin{align}
	\Phi_\ell(\xi) := \frac{1}{\ell^2}\int_{\{|y|\leq \ell\}} \e^{i\xi\cdot y}\dy.
\end{align}
Then, observe that (denoting $\xi =  \abs{\xi} (\cos \phi,\sin \phi)$ and $y = \abs{y}(\cos \theta, \sin \theta)$)
\begin{align}
\Phi_\ell(\xi) & = \frac{1}{\ell^2} \int_0^{\ell} \int_0^{2\pi} \e^{i \abs{\xi} r \cos(\theta - \phi)} r \dd\theta \dr \\
& = \frac{1}{\ell^2} \int_0^{\ell} \int_0^{2\pi} \e^{i \abs{\xi} r \cos(\theta)} r \dd\theta \dr \\
%& = \frac{1}{\ell^2} \int_0^{\ell} J_0(\abs{\xi} r) r \dr \\
& = \frac{1}{\ell^2 \abs{\xi}^2} \int_0^{\ell \abs{\xi}} J_0(z) z \dd z \\
& = \frac{1}{\ell \abs{\xi}} J_1(\ell \abs{\xi}), 
\end{align}
where $J_p(z)$ denotes the Bessel functions of the first kind (see e.g. \cite{AS}) for the identity used in the penultimate line). 
% then, 
% \begin{align}
% \int_0^{\ell \abs{\xi}} J_0(z) z dz = \ell \abs{\xi} J_1(\ell \abs{\xi}),
% \end{align}
By standard results regarding the asymptotics of Bessel functions \cite{AS} it follows that 
\begin{align}
\abs{\Phi_\ell(\xi)} & \lesssim \min\left(1, \frac{1}{\ell^{3/2}\abs{\xi}^{3/2}}\right). 
\end{align}
Hence, assuming $\ell >1$ and splitting into the regions where $ |k| > \ell^{-1/2}$ and $|k| < \ell^{-1/2}$ gives 
\begin{align}
\frac{1}{\ell^2} \abs{ \int_{\{|y|\leq \ell\}} F^\lambda (y) \dy} & \lesssim \sum_{k \in \mathbb Z^2_\lambda}\1_{\abs{k} \leq \ell^{-\frac{1}{2}}} \E|\hat{f}^\lambda(k)|^2  %+ \sum_{k \in \mathbb Z^2_\lambda :1 \geq \abs{k} \geq \delta} \frac{1}{\brak{\ell^{3/2} \delta^{3/2}}} \abs{\widehat{f}^\lambda(k)}^2 \\
 + \sum_{k \in \mathbb Z^2_\lambda}\1_{\abs{k} \geq \ell^{-\frac{1}{2}}} \frac{1}{(\ell |k|)^{3/2}} \E|\hat{f}^\lambda(k)|^2\\
 &\leqc \E\|f_{\leq \ell^{-1/2}}^\lambda\|_{\lambda}^2 + \frac{1}{\ell^{3/4}}\E\|f^\lambda\|_{\lambda}^2.
\end{align}
Taking the supremum over $\ell \in (\ell_I, \frac{1}{2}\lambda)$ yields
\[
	\sup_\lambda \sup_{\ell \in (\ell_I,\frac{1}{2}\lambda)}\frac{1}{\ell^2} \abs{ \int_{\{|y|\leq \ell\}} F^\lambda (y) \dy} \leqc \sup_{\lambda}\E\|f^\lambda_{\leq \ell_I^{-1/2}}\|_{\lambda}^2 + \frac{1}{\ell_I^{3/4}}\sup_{\lambda}\E\|f^\lambda\|_{\lambda}^2.
\]
Therefore the result follows by assumptions \eqref{c:EquiInt} and \eqref{c:apBd}.
% Given any $\zeta > 0$, by assumption \eqref{c:EquiInt} $\exists \delta$ and $\lambda_0$ such that for all $\lambda > \lambda_0$ we have
% \begin{align}
% \sum_{k \in \mathbb Z^2_\lambda : \abs{k} \leq \delta} \abs{\widehat{f}^\lambda(k)}^2 < \frac{\zeta}{2}. 
% \end{align}
% Having chosen $\delta$, by \eqref{c:apBd} we can choose $\ell_I$ such that for all $\ell > \ell_I$ we have
% \begin{align}
% \sum_{k \in \mathbb Z^2_\lambda : \abs{k} \geq \delta} \frac{1}{\brak{\ell k}^{3/2}} \abs{\widehat{f}^\lambda(k)}^2 < \frac{\zeta}{2}, 
% \end{align}
% hence the result follows. 
% It follows that,
% \begin{equation}
% \begin{aligned}
% 	\lim_{\alpha \to 0}\sup_{\ell\in (\ell_I,\ell_\alpha)}\left|\frac{1}{\ell^2}\int_{\{|y|\leq \ell\}} a_{\lambda}(y)\dy\right| &\leq \lim_{\lambda \to \infty}\sup_{\ell\in (\ell_I,\infty)}\left|\int_{\R^2} \Phi_\ell(\xi)\dee \hat K_\lambda(\xi)\right|\\
% & \lesssim \lim_{\lambda \to \infty} \sup_{\ell\in (\ell_I,\infty)}\left( \left|\int_{\abs{\xi} \leq \ell^{-1}} 1 \dee \hat K_\lambda(\xi)\right| + \frac{1}{\ell^{3/2}}\left|\int_{\abs{\xi} > \ell^{-1}} \frac{1}{\brak{\xi}^{3/2}} \dee \hat K_\lambda(\xi)\right| \right). 
% \end{aligned}
% \end{equation}
\end{proof} 

We will also need the following tensor generalization when $\{f_{\lambda}\}$ are vector valued.
\begin{lemma}\label{lem:LrgCan2}
Let $\set{f^{\lambda}}_{\lambda \geq 1}$ be a sequence of random divergence free vector fields on $\T^2_\lambda$ and define
\begin{align}
F^\lambda(y) = \EE \fint_{\TT^2_\lambda}f^\lambda(x+y)\tensor f^\lambda(x) \dx. 
\end{align}
Suppose that the following two conditions hold (uniformly in $\alpha$ and $\nu$):
\begin{subequations}
\begin{align}
\lim_{\delta \to 0} \sup_{\lambda} \EE \norm{f^\lambda_{\leq \delta}}_{\lambda}^2  = 0, \qquad \sup_{\lambda} \EE\norm{f^\lambda}_\lambda^2 < \infty. 
\end{align}
\end{subequations}
Then, 
\begin{align}
\lim_{\ell_I\to \infty} \sup_{\lambda}\sup_{\ell \in (\ell_I, \frac{1}{2}\lambda)} \frac{1}{\ell^2} \abs{ \fint_{\{|y|\leq \ell\}} y\tensor y:F^\lambda (y) \dy} = 0. 
\end{align}
\end{lemma} 
\begin{proof}
The proof is a simple modification of the previous Lemma \ref{lem:LrgCan1}. By Fourier analysis
\[
	\frac{1}{\ell^4} \int_{\{|y|\leq \ell\}} y \tensor y:F^\lambda (y) \dy = \sum_{k\in \Z^2_\lambda}\E \langle \hat{f}^\lambda(k),\Psi_\ell(k)\hat{f}^\lambda(k)\rangle,
\]
where
\[
	\Psi_\ell(\xi) = \frac{1}{\ell^4} \int_{|y|\leq \ell} y\tensor y\, \e^{i\xi\cdot y}\dy.
\]
It is important to note that $\Psi_\ell(\xi)$ can be related to $\Phi_\ell(\xi)$ from the proof of Lemma \ref{lem:LrgCan1} by
\[
	\Psi_\ell(\xi) = - \frac{1}{\ell^2}\nabla^2\Phi_\ell(\xi).
\]
Using the standard identity for Bessel functions $\frac{1}{z}\frac{\dee}{\dz}(z^{-p}J_p(z)) = z^{-p-1}J_{p+1}(z)$ allows us to deduce
\[
	\Psi_\ell(\xi) = - \frac{1}{\ell^2}\nabla^2\left(\frac{1}{\ell|\xi|}J_1(\ell|\xi|))\right) = \frac{1}{(\ell|\xi|)^2}J_2(\ell|\xi|) I - \frac{\ell^2}{(\ell|\xi|)^3}J_3(\ell|\xi|) \xi\tensor \xi.
\]
Using the fact that $f^\lambda(x)$ are divergence free and therefore $k\cdot\hat{f}^\lambda(k) =0$, we find
\[
	\langle \overline{\hat{f}^\lambda}(k),\Psi_\ell(k)\hat{f}^\lambda(k)\rangle = \frac{1}{(\ell|k|)^2}J_2(\ell|k|)|\hat{f}^\lambda(k)|^2.
\]
The proof now follows exactly as in Lemma \ref{lem:LrgCan1} using the asymptotic
\[
	\frac{1}{(\ell|\xi|)^2}|J_2(\ell|\xi|)| \leqc \min\left(1,\frac{1}{\ell^{5/2}|\xi|^{5/2}}\right).
\]
\end{proof}

\section{The direct cascade} \label{sec:Direct}
This section is devoted to the proof of part (i) of Theorem \ref{thm:InformalDual}, whose precise statement is given in the theorem below.

\begin{theorem} \label{thm:direct}
Suppose that $\lambda = \lambda(\alpha) < \infty$ is a continuous monotone increasing function such that $\displaystyle\lim_{\alpha \to 0} \lambda = \infty$.
Let $\set{u}_{\nu,\alpha > 0}$ a sequence of statistically stationary solutions such that Definition~\ref{def:WAD} holds. Then there exists $\ell_\nu \in (0,1)$ satisfying $\displaystyle\lim_{\nu \to 0} \ell_\nu = 0$ such that
\begin{subequations}\label{eq:YLlaws} 
\begin{align}
& \lim_{\ell_I \to 0}  \limsup_{\nu,\alpha \to 0}  \sup_{\ell \in [\ell_\nu, \ell_I]}\abs{\frac{1}{\ell}\EE \fint_{\mathbb S} \fint_{\TT_\lambda^2} \abs{\delta_{\ell n} \omega}^2 \delta_{\ell n} u \cdot n \, \dee x \dee n + 2\eta } = 0, \label{eq:Yaglom1}\\
& \lim_{\ell_I \to 0}  \limsup_{\nu,\alpha \to 0} \sup_{\ell \in [\ell_\nu, \ell_I]}\abs{\frac{1}{\ell^3}\EE \fint_{\mathbb S} \fint_{\TT_\lambda^2} \abs{\delta_{\ell n} u}^2 \delta_{\ell n} u \cdot n \, \dee x \dee n -\frac{1}{4}\eta } = 0, \label{eq:Lindborg1}\\
& \lim_{\ell_I \to 0}  \limsup_{\nu,\alpha \to 0}  \sup_{\ell \in [\ell_\nu, \ell_I]}\abs{\frac{1}{\ell^3}\EE \fint_{\mathbb S} \fint_{\TT_\lambda^2} \left(\delta_{\ell n} u \cdot n\right)^3 \, \dee x \dee n - \frac{1}{8}\eta } = 0. \label{eq:Lindborg3}
\end{align}
\end{subequations} 
In fact, it suffices to choose $\ell_\nu \to 0$ satisfying
\begin{align}
\sup_{\alpha \in (0,1)}\left(\nu \EE \norm{\omega}_{\lambda}^2\right)^{1/2} = o_{\nu \to 0}(\ell_\nu). \label{def:ellnu}
\end{align}
\end{theorem}
The proof of Theorem \ref{thm:direct} is split in three different subsections, one for each of the laws appearing in \eqref{eq:YLlaws}. 
The proofs of \eqref{eq:Lindborg1}--\eqref{eq:Lindborg3} require a different approach compared to the one used
to prove \eqref{eq:Yaglom1}. In particular, we will need to take advantage of certain cancellations that appear in the energy balance. 

\subsection{Proof of \eqref{eq:Yaglom1}} \label{sec:PrfYag1}
This case is the easiest and also the most similar to calculations in \cite{BCZPSW18}, so we will only sketch the main ideas. 
Recalling Lemma \ref{lem:Yag-KHM-Sphere}, we write the equation for the spherically averaged flux $\bar{\mathfrak{D}}(\ell)$ as
\begin{align}
\frac{\bar{\mathfrak{D}}(\ell)}{\ell} =-\frac{4\nu \bar{\mathfrak{B}}^\prime(\ell)}{\ell} + \frac{4\damp}{\ell^2} \int_{0}^\ell r \bar{\mathfrak{G}}(r)\dr - \frac{4}{\ell^2}\int_{0}^\ell r \bar{\mathfrak{a}}(r)\dr.\label{eq:YagBal}
% \frac{2\nu}{\pi\ell^2}\int_{|y|\leq \ell}\Delta \mathfrak{C}(y)\dy  + \frac{2\damp}{\pi\ell^2} \int_{|y|\leq\ell} \mathfrak{G}(y)\dy - \frac{2}{\pi\ell^2}\int_{|y|\leq \ell} \mathfrak{a}(y)\dy. 
\end{align}
The three terms on the RHS are considered in succession. As in \cite{BCZPSW18}, the anomalous dissipation assumptions in Definition \ref{def:WAD} show that the first two terms on the RHS drop out over a suitably defined inertial range. Similarly to \cite{BCZPSW18}, the last term converges to $2\eta$ as $\ell \to 0$.   

\smallskip

\noindent \textbf{Step 1.} Firstly, we prove that for all $\ell_\nu$ satisfying \eqref{def:ellnu}, it holds
	\begin{equation}\label{lem:VanVisYag}
\lim_{\nu \to 0} \sup_{\alpha \in (0,1)} \sup_{\ell \in (\ell_\nu, 1)}  \abs{\frac{4\nu\bar{\mathfrak{B}}^\prime(\ell)}{\ell}} = 0. 
	\end{equation}
% Note that (see \cite{BCZPSW18}), 
% \begin{align}
% 	-\frac{4\nu}{\ell^2}\int_0^\ell\left(r\bar{\mathfrak{B}}''(r)+\bar{\mathfrak{B}}'(r)\right)\dr = -4\nu \frac{\bar{\mathfrak{B}}'(\ell)}{\ell}. 
% \end{align}
% Then, 
% Applying the divergence theorem gives
% $$
% 	\abs{\frac{2\nu}{\pi\ell^2}\int_{|y|\leq \ell}\Delta \mathfrak{B}(y)\dy} = \abs{\frac{2\nu}{\pi\ell^2}\int_{|y|=\ell}\frac{y}{|y|}\cdot\nabla \mathfrak{B}(y)\dee S(y)}\leq  \frac{\nu}{\ell} \sup_{|y|=\ell}|\nabla Q(y)|. 
% $$
Using the enstrophy balance \eqref{eq:ENSTRObal} and the definition of $\bar{\mathfrak{B}}(\ell)$ in  \eqref{eq:Bbardef}
\begin{align}
\sup_{\ell\in (\ell_\nu,1)}\frac{\nu}{\ell} |\bar{\mathfrak{B}}^\prime(\ell)|& \lesssim \frac{1}{\ell_\nu} \left(\nu \EE\norm{\grad \omega}_{\lambda}^2\right)^{1/2} \left( \nu \EE\norm{\omega}_{\lambda}^2 \right)^{1/2} \lesssim \frac{\eta^{1/2}}{\ell_\nu} \left( \nu \EE\norm{\omega}_{\lambda}^2 \right)^{1/2}, 
\end{align}
which then vanishes as $\nu \to 0$ by the weak anomalous dissipation assumption~\eqref{eq:WADa}, proving \eqref{lem:VanVisYag}. 
\smallskip

\noindent \textbf{Step 2.} Secondly, we show vanishing of the damping over the inertial range, meaning we prove that
	\begin{equation}\label{eq:friction}
	\lim_{\alpha \to 0} \sup_{\nu \in (0,1)}\abs{\frac{4\damp}{\ell^2}\int_{0}^\ell r \bar{\mathfrak{G}}(r)\dr} = 0.
	\end{equation}
This follows from the observation that 
\begin{align}
\abs{\frac{2\damp}{\ell^2} \int_{0}^\ell r \bar{\mathfrak{G}}(r)\dr} \lesssim \alpha \EE\norm{(-\Delta)^{-\gamma}\omega}_{\lambda}^2, 
\end{align}
which vanishes by the weak anomalous dissipation assumption \eqref{eq:WADb}. 
\smallskip

\noindent \textbf{Step 3.}  Lastly, it remains to show that
\begin{align}
\lim_{\ell_I \rightarrow 0} \sup_{\ell \in (0,\ell_I)} \abs{\frac{4}{\ell^2}\int_{0}^\ell r \bar{\mathfrak{a}}(r)\dr-2 \eta} =0 . \label{eq:acont}
\end{align}
Since $\bar{\mathfrak{a}}(0) = \eta$, by the regularity of $\mathfrak{a}$ we have
\begin{align}
\lim_{\ell_I \rightarrow 0} \sup_{\ell \in (0,\ell_I)} \frac{1}{\ell^2}\int_{0}^\ell r \left|\bar{\mathfrak{a}}(r) - \bar{\mathfrak{a}}(0)\right| \dr = 0,
\end{align}
and \eqref{eq:acont} follows immediately.
Now, we collect \eqref{lem:VanVisYag}, \eqref{eq:friction}, \eqref{eq:acont}  and combine them with \eqref{eq:YagBal}, henceforth deducing 
\eqref{eq:Yaglom1}.

\subsection{Proof of \eqref{eq:Lindborg1}}\label{sub:Lindborg1}
The starting point is Lemma \ref{lem:S0-Sphere}, and specifically the equation
\begin{equation}
\frac{\bar{D}(\ell)}{\ell^3} =-\frac{4\nu}{\ell^3}\bar{\Gamma}^\prime(r)  + \frac{4\damp}{\ell^4} \int_0^\ell r\bar{\da}(r)\dr - \frac{4}{\ell^4}\int_0^\ell r\bar{a}(r)\dr. \label{eq:Lindborg1bal}
\end{equation}
We now analyze each of the 3 terms individually. However, the second term does not vanish (as it happened in the proof of \eqref{eq:Yaglom1}),
rather, a non-trivial cancellation takes place between the second and third term.

\smallskip

\noindent \textbf{Step 1.} We first show that
\begin{align}\label{eq:ExpDampLind1}
\frac{4}{\ell^4}\int_0^\ell r\bar{a}(r)\dr = \frac{2\eps}{\ell^2} -\frac{\eta}{4} + o_{\ell \to 0}(1). 
\end{align}
Recall $\bar{a}$ is defined by 
\begin{align}
	\bar{a}(r)  = \frac12\sum_j \fint_{\S}   \fint_{\mathbb T^2_\lambda}   g^\lambda_j(x)\cdot g^\lambda_j(x+rn) \dd x \dd n. 
\end{align}
We Taylor-expand the factor $g^\lambda_j(x+rn)$, and use that, for each fixed $j$,
\begin{align}
%\sum_{i}\fint_{\mathbb T^2_\lambda}  rn^i g^{\lambda}_j(x)\cdot\partial_{x_i} g^{\lambda}_j(x) \dd x=0,
\fint_{\mathbb T^2_\lambda}  g^{\lambda}_j(x)\cdot\partial_{x_i} g^{\lambda}_j(x) \dd x=0,
\end{align} 
and we integrate by parts to obtain 
	\begin{equation*}
	\begin{split}
	\fint_{\mathbb T^2_\lambda}   g^\lambda_j(x)\cdot g^\lambda_j(x+rn) \dd x
	&=\fint_{\mathbb T^2_\lambda}  \left( |g^\lambda_j(x)|^2 +\frac{r^2}{2} \sum_{i,m}n^i n^m g^{\lambda}_j(x)\cdot\partial_{x_m}\partial_{x_i} g^{\lambda}_j(x)\right) \dd x+  \mathcal{O}(r^3)\\
	&  =\fint_{\mathbb T^2_\lambda}  \left( |g^\lambda_j(x)|^2 -\frac{r^2}{2} \sum_{i,m}n^i n^m \partial_{x_m}g^{\lambda}_j(x)\cdot\partial_{x_i} g^{\lambda}_j(x)\right) \dd x+  \mathcal{O}(r^3).
	\end{split}
	\end{equation*}
Since	
	\begin{equation}\label{eq:trigo}
	\fint_{\S} n^i n^m \dd n=\frac{1}{2}\delta_{im},
	\end{equation}
we realize from the definitions \eqref{eq:eps} and \eqref{eq:eta} of $\eps$ and $\eta$ that
	\begin{equation*}
	\begin{split}
	\bar{a}(r)&=\frac12\sum_j\fint_{\mathbb T^2_\lambda}  |g^\lambda_j(x)|^2 \dd x -\frac{r^2}{8}\fint_{\mathbb T^2_\lambda} |\Grad g^{\lambda}_j(x)|^2 \dd x +  \mathcal{O}(r^3)= \eps-\frac{r^2}{4}\eta +  \mathcal{O}(r^3).
	\end{split}
	\end{equation*}
	Therefore, \eqref{eq:ExpDampLind1} follows by a simple integration.

\smallskip

\noindent \textbf{Step 2.} We now turn to the second term in the RHS of \eqref{eq:Lindborg1bal}, for which we want to deduce the following property:
for $\ell_\nu$ satisfying \eqref{def:ellnu} we have 	
\begin{align} \label{eq:ExpDampLind2}
& \lim_{\ell_I \to 0} \limsup_{\nu,\alpha \to 0} \sup_{\ell \in (\ell_\nu, \ell_I)}\abs{\frac{4\damp}{\ell^4} \int_0^\ell r\bar{\da}(r)\dr  - \frac{2\eps}{\ell^2}  } = 0.
\end{align}
We recall that in the above formula, the order in which the $\nu$ and $\alpha$ are taken to 0 does not matter. First notice that
\begin{align}
\frac{4\damp}{\ell^4} \int_0^\ell r\bar{\da}(r) \dee r= \frac{4\damp}{\ell^4} \int_0^\ell r \left( \bar{\da}(r)-\bar{\da}(0) \right) \dee r + \frac{2\damp}{\ell^2}\bar{\da}(0). \label{eq:Pelican}
\end{align}
In light of \eqref{eq:ENERbal} and \eqref{def:ellnu}, we have 
\begin{align}\label{eq:eafr1}
\damp \bar{\da}(0)= \eps - \nu \EE \norm{\omega}^2_\lambda=\eps+o_{\nu \to 0}(\ell_\nu^2).
%\frac{1}{\ell^2}\left(\eps - \damp \bar{\da}(0) \right) = \frac{\nu}{\ell^2} \EE \norm{\omega}^2_\lambda, 
\end{align}
Moreover, 
\begin{align}\label{eq:fsdagads}
\bar{\da}'(\ell) =\sum_{i,j} \EE\fint_{\S}\fint_{\mathbb T_\lambda^2} n^i  \partial_{x_i} (-\Delta)^{-\gamma} u^j(x+\ell n) (-\Delta)^{-\gamma} u^j(x) \dx \dee n, 
\end{align}
and hence $\bar{\da}'(0) = 0$.
Furthermore, taking a second derivative of \eqref{eq:fsdagads} and integrating by parts,  we have
\begin{align}
\sup_{\ell \in [0,1) }\abs{\bar{\da}''(\ell)}\lesssim \EE\norm{(-\Delta)^{-\gamma}\omega}^2_\lambda. 
\end{align}
As a consequence, 
\begin{align}
\abs{\bar{\da}(r) - \bar{\da}(0)} \lesssim r^2 \EE\norm{(-\Delta)^{-\gamma}\omega}^2_\lambda,
\end{align}
and in particular we deduce that
\begin{align}\label{eq:eafr2}
\frac{4\damp}{\ell^4} \int_0^\ell r \abs{\bar{\da}(r)-\bar{\da}(0)}\dr \lesssim \damp \EE\norm{(-\Delta)^{-\gamma}\omega}^2_\lambda. 
\end{align}
We now use \eqref{eq:eafr1} and \eqref{eq:eafr2} in \eqref{eq:Pelican} together with the weak anomalous dissipation assumption
\eqref{eq:WADb} to deduce \eqref{eq:ExpDampLind2}.

\smallskip

\noindent \textbf{Step 3.} Regarding the first term in \eqref{eq:Lindborg1bal}, analogous to \cite{BCZPSW18}, we have
\begin{align}\label{eq:VanVisLind1}
%\limsup_{\nu \to 0} \sup_{\alpha \in (0,1)}
\limsup_{\nu,\alpha \to 0}\sup_{\ell \in (\ell_\nu, \ell_I)} \frac{4\nu}{\ell^4}\bar{\Gamma}'(\ell) = 0. 
\end{align}
Indeed, we have
\begin{align}\label{eq:htgf}
\frac{4\nu}{\ell^3}\abs{\bar{\Gamma}'(\ell)} \lesssim \frac{1}{\ell^2} \nu \EE \norm{\omega}_{\lambda}^2, 
%-4\nu \frac{\bar{\Gamma}'(\ell)}{\ell^3} = -\frac{4\nu}{\ell^3}\EE \int_{\mathbb S^1} \grad_x u^j(x) \cdot \hat{n} ( u^j(x+\ell \hat{n}) - u^j(x)) dx. 
\end{align}
which vanishes in the way described in \eqref{eq:VanVisLind1} due to the choice \eqref{def:ellnu}. 
Collecting \eqref{eq:ExpDampLind1}, \eqref{eq:ExpDampLind2} and \eqref{eq:VanVisLind1}, we deduce \eqref{eq:Lindborg1} from 
\eqref{eq:Lindborg1bal}.

\subsection{Proof of \eqref{eq:Lindborg3}} 

Analogous to the proof of the 4/5 law in \cite{BCZPSW18} (see also \cite{Eyink2003}), we use \eqref{eq:Lindborg1} to prove \eqref{eq:Lindborg3}.
We begin with the balance given by Lemma \ref{lem:khmparallel}, which we write here for the reader's convenience as
\begin{align}\label{eq:Sparati}
\frac{\bar{D}_{||}(\ell)}{\ell^3}= - \frac{4\nu}{\ell^3} \bar{\Gamma}_{||}^\prime(\ell) +\frac{2}{\ell^6}\int_0^\ell r^2 \bar{D}(r)\dr + \frac{4\damp}{\ell^6}  \int_0^\ell r^3 \bar{\da}_{||}(r)\dee r - \frac{4}{\ell^6}\int_0^\ell r^3 \bar{a}_{||}(r) \dr.
\end{align}
The proof consists of several steps, which deal with the terms in \eqref{eq:Sparati} one by one. 

\smallskip

\noindent \textbf{Step 1.} 
We first show that for $\ell_\nu$ chosen as in \eqref{def:ellnu}, there holds
\begin{align}\label{eq:VanDissLind34}
%\lim_{\ell_I \to 0} \limsup_{\nu \to 0} \sup_{\alpha \in (0,1)}\sup_{\ell \in (\ell_\nu, \ell_I)}\abs{\frac{4\nu}{\ell^3}\bar{\Gamma}_{||}'(\ell)} = 0.
\lim_{\ell_I \to 0} \limsup_{\alpha,\nu \to 0} \sup_{\ell \in (\ell_\nu, \ell_I)}\abs{\frac{4\nu}{\ell^3}\bar{\Gamma}_{||}^\prime(\ell)} = 0. 
\end{align}
We first note that
\begin{align}
\bar{\Gamma}_{||}^\prime(\ell) = \EE \fint_{\S} \fint_{\T^2_\lambda} n^i \left(u^i(x) - u^i(x+\ell n)\right) n^k n^j \partial_{x_k} u^j(x+n\ell) \dee x \dee n,  
\end{align}
so that
\begin{align}\label{eq:HHHHH}
\abs{\bar{\Gamma}_{||}^\prime(\ell)} \lesssim \ell \EE \norm{\omega}_{\lambda}^2.
\end{align}
The claim \eqref{eq:VanDissLind34} now follows from \eqref{def:ellnu}.

\smallskip

\noindent \textbf{Step 2.} For $\ell_\nu$ satisfying \eqref{def:ellnu} we have
		\begin{align}
		\lim_{\ell_I \to 0} \limsup_{\nu,\alpha \to 0} \sup_{\ell \in (\ell_\nu, \ell_I)}\abs{\frac{4\damp}{\ell^6}  \int_0^\ell r^3 \bar{\da}_{||}(r)\dee r - \frac{\eps}{2\ell^2}  } = 0. \label{eq:ExpDampLind34}
		\end{align}

Again to show this, we write
\begin{align}
\frac{4\damp}{\ell^6} \int_0^\ell r^3 \bar{\da}_{||}(r) \dee r= \frac{4\damp}{\ell^6} \int_0^\ell r^3 \left( \bar{\da}_{||}(r)-\bar{\da}_{||}(0) \right) \dee r + \frac{\damp \bar{\da}_{||}(0)}{\ell^2}. \label{eq:DampLind1}
\end{align}
Note that by \eqref{eq:trigo}
\[
	\damp\bar{\da}_{||}(0) = \damp \left(\fint_{\S} n\tensor n \,\dn\right):\da(0) = \frac{\damp}{2}\E\|(-\Delta)^{-\gamma}u\|^2_\lambda = \frac{\ep}{2} - o_{\nu\to 0}(\ell_\nu^2).
\]
On the other hand since $\bar{G}_{||}^\prime(0) = 0$, $\bar{\da}_{||}$ can be estimated similarly to $\bar{\da}$ (see \eqref{eq:fsdagads} and the subsequent computations), and we infer that
	\begin{align}
	\sup_{r\in [0,1)} |\bar{\da}''_{||}(r)|\lesssim \EE\norm{(-\Delta )^{-\gamma} \omega}_{\lambda}^2.
	\end{align}
	Therefore
	\begin{equation}
	\left|\frac{4\damp}{\ell^6} \int_0^\ell r^3 \left( \bar{\da}_{||}(r)-\bar{\da}_{||}(0) \right) \dee r\right|\lesssim \damp \EE\norm{(-\Delta )^{-\gamma} \omega}_{\lambda}^2,
	\end{equation}
	and \eqref{eq:ExpDampLind34} follows  from \eqref{eq:WADb}.

\smallskip

\noindent \textbf{Step 3.} 
There holds
\begin{align}\label{eq:S0ExpLind34}
\lim_{\ell_I \to 0} \limsup_{\nu,\alpha \to 0} \sup_{\ell \in (\ell_\nu,\ell_I)} \abs{\frac{2}{\ell^6} \int_0^\ell r^2 \bar{D}(r) \dr - \frac{\eta}{12} } = 0. 
%\frac{2}{\ell^6} \int_0^\ell r^2 S_0(r) \dr= \frac{\eta}{12} + o_{\ell,\nu,\damp\to 0}(1).
\end{align}
Indeed, this simply follows from \eqref{eq:Lindborg1} proved earlier in Section \ref{sub:Lindborg1} followed by integration. 
% , we in fact proved that 
% \begin{align}
% 	\frac{S_0(r)}{r^3}=\frac{\eta}{4} +o_{r,\nu,\damp\to 0}(1),
% \end{align}
% so that \eqref{eq:S0ExpLind34} follows by integration.

\smallskip

\noindent \textbf{Step 4.} For the energy input, there holds

\begin{align}\label{eq:InputExLind3}
\frac{4}{\ell^6}  \int_0^\ell r^3 \bar{a}_{||}(r)\dee r =\frac{\eps}{2\ell^2} -\frac{\eta}{24} + o_{\ell\to 0}(1).
\end{align}
Recall that 
\[
	\bar{a}_{||}(r) = \frac{1}{2}\sum_j\fint_{\S}\fint_{\T^2_\lambda} (g^\lambda_j(x)\cdot n)(g^\lambda_j(x+r n)\cdot n)\dx \dn.
\]
Using Taylor's expansion as in Step 1 of Section \ref{sub:Lindborg1} and integration by parts, we have	
	\begin{align}\label{eq:faen}
	&\fint_{\mathbb T^2_\lambda}   g^\lambda_j(x)\cdot n\, g^\lambda_j(x+rn)\cdot n \dd x\notag\\
	&\qquad=\fint_{\mathbb T^2_\lambda}  \left( (g^\lambda_j(x)\cdot n)^2 -\frac{r^2}{2} \sum_{k,i,m,p}n^i n^m n^k n^p\partial_{x_m}g^{\lambda,(k)}_j(x)\partial_{x_i} g^{\lambda,(p)}_j(x)\right) \dd x+  \mathcal{O}(r^3).
	\end{align}
	Using \eqref{eq:trigo}, the first term yields
	\begin{align}\label{eq:a1}
	\sum_j\frac{2}{\ell^6}\int_0^\ell r^3 \fint_{\S}\fint_{\TT^2_\lambda} (g^\lambda_j(x)\cdot n)^2 \dd x \dd n \dr = 	\frac{1}{\ell^6}\int_0^\ell r^3\fint_{\TT^2_\lambda} |g^\lambda_j(x)|^2 \dd x \dr = \frac{\eps}{2\ell^2}.
	\end{align}
For the second term, we use that	
	\begin{align}\label{eq:trigo2}
	\fint_{\S}n^in^mn^kn^p \dee n = \frac{1}{8}(\delta_{i,m}\delta_{k,p} + \delta_{i,k}\delta_{m,p} + \delta_{i,p}\delta_{m,k}),
	\end{align}
 and the fact that $g_j^\lambda$ is divergence free to deduce the identity
	\begin{equation}
	\sum_{k,i,m,p}\left(\fint_{\S}n^i n^m n^k n^p\dee n\right)\partial_{x_m}g^{\lambda,(k)}_j\partial_{x_i} g^{\lambda,(p)}_j = \frac{1}{8}|\nabla g^\lambda_j|^2 + \sum_{i,k}\frac{1}{8}\partial_{x_k}g^{\lambda,(i)}_j\partial_{x_i}g^{\lambda,(k)}_j. 
	\end{equation}	
	Again using the divergence free property of $g^\lambda_j$, integration by parts reveals
	\[
	\sum_{i,k}\int_{\T^2_\lambda} \partial_{x_k}g^{\lambda,(i)}_j(x)\partial_{x_i}g^{\lambda,(k)}_j(x)\dx =0.
	\]
Consequently we obtain
	\begin{align}
	\sum_{k,i,m,p}\left(\fint_{\S}n^i n^m n^k n^p\dd n\right)\left(\fint_{\T_\lambda^2}\partial_{x_m}g^{\lambda,(k)}_j(x)\partial_{x_i} g^{\lambda,(p)}_j(x)\dx\right)  = \frac{1}{8}\fint_{\TT^2_\lambda}\left|\Grad g_j^\lambda(x)\right|^2 \dd x, 
	\end{align}		
	and after summing and integrating, we get
	\begin{align}\label{eq:a2}
	-\sum_j \frac{2}{\ell^6}\int_{0}^\ell\frac{r^5}{2}	\fint_{\S}\fint_{\TT^2_\lambda}\sum_{k,\ell,m,p}n^\ell n^m n^k n^p\partial_{x_m}g^{\lambda,(k)}_j(x)\partial_{x_\ell} g^{\lambda,(p)}_j(x) \dd n \dd x \dr = -\frac{\eta}{24}. 
	\end{align}
From \eqref{eq:a1} and  \eqref{eq:a2}, our claim \eqref{eq:InputExLind3} follows.
Putting together \eqref{eq:VanDissLind34}, \eqref{eq:ExpDampLind34},  \eqref{eq:S0ExpLind34}, and \eqref{eq:InputExLind3} with \eqref{eq:Sparati}
 completes the proof of \eqref{eq:Lindborg3}. 
\section{The inverse cascade} \label{sec:Inverse}
This section is devoted to the proof of part (ii) of Theorem \ref{thm:InformalDual}, whose precise statement is given in the theorem below.
\begin{theorem} \label{thm:inverse}
Suppose that $\lambda = \lambda(\alpha) < \infty$ is a continuous monotone increasing function such that $\displaystyle\lim_{\alpha \to 0} \lambda = \infty$.
Let $\set{u}_{\nu,\alpha > 0}$ a sequence of statistically stationary solutions such that Definition~\ref{def:WAD} holds. Then there exists 
$\ell_\alpha \in (1,\lambda)$ satisfying $\displaystyle\lim_{\alpha \to 0} \ell_\alpha = \infty$ such that
\begin{subequations} \label{eq:inverselaws} 
\begin{align}
& \lim_{\ell_I \to \infty}  \limsup_{\nu,\alpha \to 0}  \sup_{\ell \in [\ell_I, \ell_\alpha]}\abs{\frac{1}{\ell}\EE \fint_{\mathbb S} \fint_{\TT_\lambda^2} \abs{\delta_{\ell n} u}^2 \delta_{\ell n} u \cdot n \, \dee x \dee n - 2\eps } = 0, \label{eq:34invCas1} \\ 
& \lim_{\ell_I \to \infty}  \limsup_{\nu,\alpha \to 0}  \sup_{\ell \in [\ell_I, \ell_\alpha]}\abs{\frac{1}{\ell}\EE \fint_{\mathbb S} \fint_{\TT_\lambda^2} \left(\delta_{\ell n} u \cdot n\right)^3 \, \dee x \dee n - \frac{3}{2}\eps } = 0. \label{eq:45invCas1} 
\end{align}
\end{subequations} 
In fact, it suffices to choose $\ell_\alpha \to \infty$ satisfying
\begin{align}
\ell^2_\alpha = o \left( \left(\sup_{\nu \in (0,1)}\alpha \EE\norm{(-\Delta)^{-\gamma}\omega}_{\lambda}^2\right)^{-1}\right).  \label{def:ellalpha}
\end{align}
\end{theorem}
As in Theorem \ref{thm:direct}, also for \eqref{eq:inverselaws} the order in which we take limits is irrelevant. As before, we will
split the proof of the two statements above in two different subsections.

\subsection{Proof of \eqref{eq:34invCas1}}

We recall once more Lemma \ref{lem:S0-Sphere}, specifically that
\begin{equation}
\frac{\bar{D}(\ell)}{\ell} =-\frac{4\nu\bar{\Gamma}^\prime(\ell)}{\ell} + \frac{4\damp}{\ell^2} \int_0^\ell r\bar{\da}(r)\dr - \frac{4}{\ell^2}\int_0^\ell r\bar{a}(r)\dr. \label{eq:Lindborg1bal2}
\end{equation}
The most interesting contributions are from the energy input term and the large-scale damping. Indeed,
the contribution of the noise in fact \emph{vanishes}, in stark contrast to how the proofs in the direct cascade have proceeded.
\smallskip

\noindent \textbf{Step 1.} There holds
  \begin{align}\label{eq:gsadfdvxz}
	\lim_{\ell_I\to \infty}\lim_{\alpha \to 0}\sup_{\ell\in (\ell_I,\ell_\alpha)}\frac{1}{\ell^2}\int_0^\ell r\bar{a}(r)\dr= 0.
	\end{align}
This is an immediate consequence of Assumption~\ref{cond:g} and Lemma~\ref{lem:LrgCan1}.
\smallskip

\noindent \textbf{Step 2.} For  $\ell_\alpha$ satisfying \eqref{def:ellalpha}, we have
\begin{align}\label{eq:gkasrdbv}
\limsup_{\nu,\alpha \to 0} \sup_{\ell \in (1, \ell_\alpha)}\abs{\frac{4\damp}{\ell^2} \int_0^\ell r\bar{\da}(r)\dr - 2\eps }.	
\end{align}
Analogously to \eqref{eq:DampLind1}, we have
\begin{align}
\frac{4\damp}{\ell^2} \int_0^\ell r\bar{\da}(r) \dee r= \frac{4\damp}{\ell^2} \int_0^\ell r \left( \bar{\da}(r)-\bar{\da}(0) \right) \dee r + 2\damp\bar{\da}(0). \label{eq:DampLind11}
\end{align}
Thanks to \eqref{eq:ENERbal} and  Definition \ref{def:WAD}, we have that
\begin{align}\label{eq:graedfs}
\lim_{\nu\to0}\sup_{\alpha\in(0,1)}\abs{\damp\bar{\da}(0) - \eps}  =0.
\end{align}
Moreover, since
\begin{align}\label{eq:boundlipsda}
|\bar{\da}(r)-\bar{\da}(0)|&= \left|\E\fint_{\T^2_\lambda}\fint_{\mathbb{S}^1} (-\Delta)^{-\gamma} (u(x+rn)-u(x))\cdot (-\Delta)^{-\gamma} u(x) \dee n \dee x\right| \notag \\
&\lesssim r\left(\E\norm{(-\Delta)^{-\gamma} \omega}^2_{\lambda} \right)^{1/2}\left(\E\norm{(-\Delta)^{-\gamma}u}_{\lambda}^2 \right)^{1/2},
\end{align}
we deduce from \eqref{eq:ENERbal} that
\begin{align}
\abs{\frac{4\damp}{\ell^2} \int_0^\ell r(\bar{\da}(r)-\bar{\da}(0))\dr} 
&\lesssim \ell\left(\alpha\E\norm{(-\Delta)^{-\gamma}\omega}^2_{\lambda}\right)^{1/2}\left(\alpha\E\norm{(-\Delta)^{-\gamma}u}_{\lambda}^2\right)^{1/2} \notag\\
&\lesssim \ell\left(\alpha\E\norm{(-\Delta)^{-\gamma}\omega}^2_{\lambda}\right)^{1/2}. 
\end{align}
The claim \eqref{eq:gkasrdbv} now follows from \eqref{eq:graedfs}, the above estimate and the choice \eqref{def:ellalpha}.

\smallskip

\noindent \textbf{Step 3.}
Finally, it is straightforward to show that the effect of the viscous dissipation vanishes in the large-scale inertial range:
	\begin{equation}\label{eq:rebgv}
%	\lim_{\nu \to 0}\sup_{\alpha \in (0,1)}\frac{4\nu}{\ell^2}\int_0^\ell\left(r\bar{\Gamma}''(r)+\bar{\Gamma}'(r)\right)\dr = 0.
	\limsup_{\nu,\alpha \to 0}\frac{4\nu\bar{\Gamma}^\prime(\ell)}{\ell} = 0.
	\end{equation}
Indeed, by the same argument as in \eqref{eq:htgf}, we have
\begin{align}
\frac{4\nu}{\ell}\abs{\bar{\Gamma}'(\ell)} & \lesssim \nu\EE \norm{\omega}_{\lambda}^2,
\end{align}
and \eqref{eq:rebgv} follows from  Definition \ref{def:WAD}.
Putting together \eqref{eq:gsadfdvxz}, \eqref{eq:gkasrdbv} and  \eqref{eq:rebgv} with \eqref{eq:Lindborg1bal2}, we finish the proof of~\eqref{eq:34invCas1}.

\subsection{Proof of \eqref{eq:45invCas1}}
As in the proof of \eqref{eq:Lindborg3}, the starting point is Lemma \ref{lem:khmparallel} and the identity
\begin{align}\label{eq:Sparati2}
\frac{\bar{D}_{||}(\ell)}{\ell}= - \frac{4\nu\bar{\Gamma}^\prime_{||}(\ell)}{\ell} +\frac{2}{\ell^4}\int_0^\ell r^2 \bar{D}(r)\dr + \frac{4\alpha}{\ell^4}\int_{0}^\ell r^3\bar{\da}_{||}(r)\dr - \frac{4}{\ell^4}\int_0^\ell r^3 \bar{a}_{||}(r)\dr.
\end{align}
Notice that thanks to \eqref{eq:HHHHH},
\begin{align}\label{eq:4441}
\frac{\nu|\bar{\Gamma}^\prime_{||}(\ell)|}{\ell}  \leqc \nu \E\|\omega\|_{\lambda}^2,
\end{align}
and therefore this term can be taken care of by using \eqref{eq:WADa}. Moreover, in light of  \eqref{eq:34invCas1} we have
\begin{align}\label{eq:329bis}
\lim_{\ell_I \to \infty}  \limsup_{\nu,\alpha \to 0}  \sup_{\ell \in [\ell_I, \ell_\alpha]}\left|\frac{2}{\ell^4}\int_0^\ell r^2 \bar{D}(r)\dr - \ep\right| = 0.
\end{align}
To control the third term in the right hand side of \eqref{eq:Sparati2}, we argue as in \eqref{eq:DampLind11} and use \eqref{eq:ENERbal} to obtain
	\begin{equation}
	\frac{4\damp}{\ell^4}\int_0^\ell r^3 \bar{G}_{||}(r)\dr	= \frac{\eps}{2}+\frac{4\damp}{\ell^4}\int_0^\ell r^3 (\bar{\da}_{||}(r)-\bar{\da}_{||}(0)) \dr - \frac{\nu}{2}\E\|\nabla u\|_{\lambda}^2.
	\end{equation}
Now, as in \eqref{eq:boundlipsda},
\begin{equation}
\alpha|\bar{\da}_{||}(r)-\bar{\da}_{||}(0)| \leqc r\left(\alpha \E\|(-\Delta)^{-\gamma} \omega\|_{\lambda}^2\right)^{1/2},
\end{equation}
and therefore
\begin{equation}
\begin{aligned}
\frac{4\damp}{\ell^4} \left|\int_0^\ell r^3 (\bar{\da}_{||}(r)-\bar{\da}_{||}(0)) \dr\right| \lesssim\ell \left(\alpha \E\|(-\Delta)^{-\gamma/2} \omega\|_{\lambda}^2\right)^{1/2}. 
\end{aligned}
\end{equation}
As a consequence of \eqref{eq:WADa}--\eqref{eq:WAD2b} and \eqref{def:ellalpha}, we conclude that
\begin{align}\label{eq:329tris}
\lim_{\ell_I \to \infty}  \limsup_{\nu,\alpha \to 0}  \sup_{\ell \in [\ell_I, \ell_\alpha]}\left|\frac{4\damp}{\ell^4}\int_0^\ell r^3 \bar{G}_{||}(r)\dr- \frac{\ep}{2}\right| = 0.
\end{align}
Lastly, the claim that
\begin{align}\label{eq:329quadris}
\lim_{\ell_I \to \infty}  \limsup_{\nu,\alpha \to 0}  \sup_{\ell \in [\ell_I, \ell_\alpha]}\left|\frac{4}{\ell^4}\int_0^\ell r^3 \bar{a}_{||}(r)\dr\right| = 0,
\end{align}
follows from the fact that
\[
	\frac{1}{\ell^4}\int_0^\ell r^3 \bar{a}_{||}(r)\dr = \frac{1}{2 \ell^2}\fint_{\{|y|\leq \ell\}} y\tensor y : a(y)\dy
\] 
and an application of Lemma~\ref{lem:LrgCan2}, the divergence-free property of the noise and Assumption~\ref{cond:g}.
Thus,  \eqref{eq:45invCas1} is a consequence of \eqref{eq:4441}, \eqref{eq:329bis}, \eqref{eq:329tris} and \eqref{eq:329quadris}, together with
\eqref{eq:Sparati2}. The proof is concluded.

\section{Necessary conditions for the dual cascade} \label{sec:Nec}

\subsection{Proof of~\eqref{eq:necessary1}}
	As in the proof of~\eqref{eq:yag}, we consider the identity from Lemma~\ref{lem:Yag-KHM-Sphere},
	\begin{equation}
	\frac{\bar{\mathfrak{D}}(\ell)}{\ell} =- 2\nu\fint_{|y|\leq \ell}\Delta \mathfrak{B}(y)\dy  + 2\damp\fint_{|y|\leq\ell} \mathfrak{G}(y)\dy - 2\fint_{|y|\leq \ell} \mathfrak{a}(y)\dy. \label{eq:NecBal1}
	\end{equation}
	The first term on the right hand side of \eqref{eq:NecBal1} can be written as 
	\begin{equation*}
	\begin{split}
	- 2\nu\fint_{|y|\leq \ell}\Delta \mathfrak{B}(y)\dy & = 2\nu\fint_{|y|\leq \ell}\EE\fint_{\TT_\lambda^2}\Grad\omega(x) \cdot \Grad\omega(x+y) \dx \dy \\
	% & =  2\nu\fint_{|y|\leq \ell}\EE\fint_{\TT_\lambda^2}|\Grad\omega|^2 \dx \dy + \frac{2\nu}{\pi\ell^2}\int_{|y|\leq \ell}\EE\fint_{\TT_\lambda^2}\Grad\omega(x)\Grad(\omega(x+y)-\omega(x)) \dx \dy\\
& = 2\nu\E\norm{\Grad\omega}_\lambda^2 + 2\nu\fint_{|y|\leq \ell}\EE\fint_{\TT_\lambda^2}\Grad\omega(x)\cdot \delta_y\Grad\omega(x) \dx \dy.
	\end{split}
	\end{equation*}
	The second term then vanishes by the assumption in \eqref{def:PCdiss}
	\begin{equation*}
	\left|2\nu\fint_{|y|\leq \ell}\EE\fint_{\TT_\lambda^2}\Grad\omega(x)\delta_y\Grad\omega(x) \dx \dy\right|\lesssim \sqrt{\nu\eta}\left(\sup_{|h|\leq \ell}\E\norm{\delta_h \Grad\omega}_\lambda^2\right)^{1/2} = o_{\ell\to 0}(1),
	\end{equation*}
	uniformly in $\damp,\nu>0$ and hence the first term in \eqref{eq:NecBal1} satisfies
	\begin{equation*}
	- 2\nu\fint_{|y|\leq \ell}\Delta \mathfrak{B}(y)\dy  = 2\nu\E\norm{\Grad\omega}_\lambda^2  + o_{\ell\to 0}(1).
	\end{equation*}
	The second term in \eqref{eq:NecBal1} is estimated as follows:
	\begin{equation*}
	\begin{split}
	2\damp \fint_{|y|\leq\ell}\mathfrak{G}(y)\dy & = 	2\damp\EE\fint_{|y|\leq\ell} \fint_{\TT^2_\lambda}(-\Delta)^{-\gamma}\omega(x)(-\Delta)^{-\gamma}\omega(x+y)\dx \dy\\
	% & = 	2\damp\EE\fint_{|y|\leq\ell} \fint_{\TT^2_\lambda}|(-\Delta)^{-\gamma}\omega(x)|^2\dx \dy\\
	% &\hphantom{=} + 	2\damp \EE\fint_{|y|\leq\ell} \fint_{\TT^2_\lambda}(-\Delta)^{-\gamma}\omega(x)(-\Delta)^{-\gamma}(\omega(x)-\omega(x+y))\dx \dy\\
	& = 2\damp \EE\norm{(-\Delta)^{-\gamma}\omega}_{\lambda}^2 + 	2\damp\EE\fint_{|y|\leq\ell} \fint_{\TT^2_\lambda}(-\Delta)^{-\gamma}\omega(x)(-\Delta)^{-\gamma}\delta_y\omega(x)\dx \dy.
	\end{split}
	\end{equation*}
	By the assumption in \eqref{def:PCdamp}, we have the following (uniformly in $\damp,\nu$)
	\begin{equation*}
	\left| 2\damp\EE\fint_{|y|\leq\ell} \fint_{\TT^2_\lambda}(-\Delta)^{-\gamma}\omega(x)(-\Delta)^{-\gamma}\delta_y\omega(x)\dx \dy\right|\lesssim \sqrt{\damp\eta}\left(\sup_{|h|\leq \ell}\E\norm{(-\Delta)^{-\gamma}\delta_h\omega}_\lambda^2\right)^{1/2} =o_{\ell\to 0}(1),
	\end{equation*}
	and hence
	\begin{equation*}
	2\damp\fint_{|y|\leq\ell} \mathfrak{G}(y)\dy =2\damp \EE\norm{(-\Delta)^{-\gamma}\omega}_{\lambda}^2  + o_{\ell\to 0}(1).
	\end{equation*}
	Combining with~\eqref{eq:acont}, we obtain~\eqref{eq:necessary1} from \eqref{eq:NecBal1}.

\subsection{Proof of~\eqref{eq:necessary2}}
	As in the proof of~\eqref{eq:Lindborg1}, we use Lemma~\ref{lem:S0-Sphere} to obtain equation~\eqref{eq:Lindborg1bal}
	\begin{equation}
	\frac{\bar{D}(\ell)}{\ell^3} =-\frac{4\nu\bar{\Gamma}^\prime(\ell)}{\ell^3}  + \frac{4\damp}{\ell^4} \int_0^\ell r\bar{\da}(r)\dr - \frac{4}{\ell^4}\int_0^\ell r\bar{a}(r)\dr, \label{eq:NecBal2}
	\end{equation}
	and then consider each of the terms on the right hand side separately. 
	
\smallskip

\noindent \textbf{Step 1.} The first term on the right-hand side of \eqref{eq:NecBal2} satisfies (c.f.\eqref{eq:htgf})
	\begin{equation}
	\frac{4\nu}{\ell^3}\bar{\Gamma}'(\ell) = \frac{4\nu}{\ell^3}\bar{\Gamma}'(0) + \frac{4\nu}{\ell^2}\bar{\Gamma}''(0)+\frac{2\nu}{\ell}\bar{\Gamma}'''(0)+\frac{2\nu}{3}\bar{\Gamma}''''(\vartheta), \label{eq:GammaExpand}
	\end{equation}
	for some $\vartheta\in [0,\ell]$. We have  %but we have to show that this Taylor expansion is justified and compute all the terms. We have
	\begin{equation*}
	|\bar{\Gamma}'(r)| =\left|\sum_{i,j}\EE\fint_{\S}\fint_{\TT_{\lambda}^2} n^i \partial_{x_i} u^j(x+r n)u^j(x)\dx \dee n\right|\leq (\EE\norm{\Grad u}_\lambda^2)^{1/2} (\EE\norm{u}_\lambda^2)^{1/2},
	\end{equation*}
	which is bounded for any fixed $\nu,\damp>0$, and therefore
	\begin{equation*}
	\lim_{r\to 0}	\bar{\Gamma}'(r) =\sum_{i,j}\EE\fint_{\S}\fint_{\TT_{\lambda}^2} n^i \partial_{x_i} u^j(x)u^j(x)\dx \dee n =0.
	\end{equation*}
Similarly,
	\begin{equation*}
	%	\begin{split}
	\bar{\Gamma}''(r)
	%& =\sum_{i,j,k}\EE\fint_{\S}\fint_{\TT_{\lambda}^2} n^i n^k\partial_{x_i}\partial_{x_k} u^j(x+r n)u^j(x)\dx \dee n\\& 
	= -\sum_{i,j,k}\EE\fint_{\S}\fint_{\TT_{\lambda}^2} n^i n^k\partial_{x_k} u^j(x+r n)\partial_{x_i} u^j(x)\dx \dee n,
	%	\end{split}
	\end{equation*}
	and hence $|\bar{\Gamma}''(r)|\lesssim \EE\norm{\Grad u}_\lambda^2$. Moreover, by~\eqref{eq:trigo}, we have
	\begin{equation*}
	\bar{\Gamma}''(0)	 = -\sum_{i,j,k}\EE\fint_{\S}\fint_{\TT_{\lambda}^2} n^i n^k\partial_{x_k} u^j(x)\partial_{x_i} u^j(x)\dx \dee n
	=-\frac{1}{2}\EE\fint_{\TT_{\lambda}^2} |\Grad u(x)|^2\dx. 
	\end{equation*}
For the third derivative, 
	\begin{equation*}
	%	\begin{split}
	\bar{\Gamma}'''(r)
	%& =\sum_{i,j,k,m}\EE\fint_{\S^{1}}\fint_{\TT_{\lambda}^2} n^i n^k n^m\partial_{x_m}\partial_{x_i}\partial_{x_k} u^j(x+r n)u^j(x)\dx \dee n\\& 
	= -\sum_{i,j,k,m}\EE\fint_{\S}\fint_{\TT_{\lambda}^2} n^i n^k n^m\partial_{x_m}\partial_{x_k} u^j(x+r n)\partial_{x_i} u^j(x)\dx \dee n,
	%	\end{split}
	\end{equation*}
	hence $|\bar{\Gamma}'''(r)|\lesssim (\EE\norm{\Grad u}_\lambda^2)^{1/2}(\EE\norm{\Grad \omega}_\lambda^2)^{1/2}$, and for $r=0$, $\bar{\Gamma}'''(0)=0$ since
	\begin{equation}\label{eq:trigo3}
	\fint_{\S}n^in^jn^k\,\dee n = 0.
	\end{equation}
	For the fourth derivative of $\Gamma$, we have
	\begin{equation*}
	\begin{split}
	\bar{\Gamma}''''(r)
	%& =\sum_{i,j,k,m,q}\EE\fint_{\S^{1}}\fint_{\TT_{\lambda}^2} n^i n^k n^m n^q\partial_{x_m}\partial_{x_i}\partial_{x_k}\partial_{x_q} u^j(x+r n)u^j(x)\dx \dee n\\& 
	&= \sum_{i,j,k,m,q}\EE\fint_{\S}\fint_{\TT_{\lambda}^2} n^i n^k n^m n^q\partial_{x_m}\partial_{x_k} u^j(x+r n)\partial_{x_i} \partial_{x_q}u^j(x)\dx \dee n\\
	& = \sum_{i,j,k,m,q}\EE\fint_{\S}\fint_{\TT_{\lambda}^2} n^i n^k n^m n^q\partial_{x_m}\partial_{x_k} u^j(x)\partial_{x_i} \partial_{x_q}u^j(x)\dx \dee n\\
	&\hphantom{=}+ \sum_{i,j,k,m,q}\EE\fint_{\S}\fint_{\TT_{\lambda}^2} n^i n^k n^m n^q\partial_{x_m}\partial_{x_k}\delta_{rn} u^j(x)\partial_{x_i} \partial_{x_q}u^j(x)\dx \dee n\\
	&=: I_1 + I_2.
	\end{split}
	\end{equation*}
	For $I_1$, we use ~\eqref{eq:trigo2} to conclude
	\[
	I_1 = \frac{3}{8}\EE\fint_{\TT_{\lambda}^2} |\Grad\omega(x)|^2\dx.
	\]
	For $I_2$ we use instead 
	\begin{equation*}
	\left|I_2\right|\lesssim (\EE\norm{\Grad\omega}_\lambda^2)^{1/2} \left(\sup_{|h|\leq r}\EE\norm{\delta_h\Grad\omega}_\lambda^2\right)^{1/2}.
	\end{equation*}
	Combining the last few calculations with \eqref{eq:GammaExpand}, we obtain
	\begin{equation}\label{eq:dritt1}
	\frac{4\nu\bar{\Gamma}^\prime(\ell)}{\ell^3}  =   -\frac{2\nu}{\ell^2}\EE\norm{\Grad u}_\lambda^2+\frac{\nu}{4}\E\norm{\Grad\omega}^2_\lambda + \mathcal{O}\left( \left(\eta\nu\sup_{|h|\leq \ell}\EE\norm{\delta_h\Grad\omega}_\lambda^2\right)^{1/2}\right).
	\end{equation}
	
\smallskip

\noindent \textbf{Step 2.} Next, we estimate the second term on the right-hand side of \eqref{eq:NecBal2}. By Taylor expansion,
	\begin{equation}\label{eq:taylorx}
	\frac{4\damp}{\ell^4} \int_0^\ell r\bar{\da}(r)\dr  = \frac{2\damp}{\ell^2}\bar{\da}(0) + \frac{4\damp}{3\ell}\bar{\da}'(0) + \frac{2\damp}{\ell^4}\int_0^\ell r^3\bar{\da}''(\vartheta_r)\dr,
	\end{equation}
	for some $\vartheta_r\in [0,r]$, $r\in [0,\ell]$. We have already computed the first two terms in~\eqref{eq:eafr1} and~\eqref{eq:fsdagads} and found in particular that $\bar{\da}(0) = \E\norm{(-\Delta)^\gamma u}_\lambda^2$ and $\bar{\da}'(0)=0$. For the last term, we have 
	%(integrating by parts)
	\begin{equation*}
	\begin{split}
	\bar{\da}''(r) &=-\sum_{i,j,k} \EE\fint_{\S }\fint_{\mathbb T_\lambda^2} n^i n^k  \partial_{x_i} (-\Delta)^{-\gamma} u^j(x+r n) \partial_{x_k}(-\Delta)^{-\gamma} u^j(x) \dx \dee n\\
	& =  -\sum_{i,j,k} \EE\fint_{\S }\fint_{\mathbb T_\lambda^2} n^i n^k  \partial_{x_i} (-\Delta)^{-\gamma} u^j(x) \partial_{x_k}(-\Delta)^{-\gamma} u^j(x) \dx \dee n  \\
	&\hphantom{=} - \sum_{i,j,k} \EE\fint_{\S}\fint_{\mathbb T_\lambda^2} n^i n^k  \partial_{x_i} (-\Delta)^{-\gamma} u^j(x) \partial_{x_k}(-\Delta)^{-\gamma}\delta_{rn} u^j(x) \dx \dee n\\
	& =:  I_1   + I_2
	\end{split}
	\end{equation*}
    Once more~\eqref{eq:trigo} implies
    \[
    I_1 = -\frac{1}{2}\EE\norm{(-\Delta)^{-\gamma}\omega}^2_\lambda,
    \]
    while we also obtain
	\begin{equation*}
	\left|I_2\right| \lesssim \left(\EE \norm{(-\Delta)^{-\gamma}\omega}_\lambda^2\right)^{1/2} \left(\sup_{|h|\leq r} \EE \norm{\delta_h(-\Delta)^{-\gamma}\omega}_\lambda^2\right)^{1/2}.
	\end{equation*}
	Combining with~\eqref{eq:taylorx}, we have
	\begin{equation}\label{eq:dampingestimate}
	\frac{4\damp}{\ell^4} \int_0^\ell r\bar{\da}(r)\dr  = \frac{2\damp}{\ell^2} \E\norm{(-\Delta)^\gamma u}_\lambda^2  -\frac{\damp}{4}\EE \norm{(-\Delta)^{-\gamma}\omega}_\lambda^2 + \mathcal{O}\left( \left(\damp\eta\sup_{|h|\leq \ell} \EE \norm{\delta_h(-\Delta)^{-\gamma}\omega}_\lambda^2\right)^{1/2}\right). 
	\end{equation}
	
\smallskip

\noindent \textbf{Step 3.}
	Now combining~\eqref{eq:dritt1}, ~\eqref{eq:dampingestimate}, and ~\eqref{eq:ExpDampLind1} and inserting in~\eqref{eq:NecBal2}, we obtain
	\begin{equation*}
	\begin{split}
	\frac{\bar{D}(\ell)}{\ell^3} &=\frac{2}{\ell^2}\left(\nu\EE\norm{\Grad u}_\lambda^2+\damp\E\norm{(-\Delta)^\gamma u}_\lambda^2-\eps\right)+\frac{1}{4}\left(\eta - \nu\E\norm{\Grad\omega}^2_\lambda-{\damp}\EE \norm{(-\Delta)^{-\gamma}\omega}_\lambda^2 \right)\\
	&\quad+ \mathcal{O}\left( \left(\damp\eta\sup_{|h|\leq \ell} \EE \norm{\delta_h(-\Delta)^{-\gamma}\omega}_\lambda^2\right)^{1/2}+ \left(\eta\nu\sup_{|h|\leq \ell}\EE\norm{\delta_h\Grad\omega}_\lambda^2\right)^{1/2}\right) + o_{\ell \to 0}(1) \\
	& =\mathcal{O}\left( \left(\damp\eta\sup_{|h|\leq \ell} \EE \norm{\delta_h(-\Delta)^{-\gamma}\omega}_\lambda^2\right)^{1/2}+ \left(\eta\nu\sup_{|h|\leq \ell}\EE\norm{\delta_h\Grad\omega}_\lambda^2\right)^{1/2}\right)
	+ o_{\ell \to 0}(1),
	\end{split}
	\end{equation*}
	where we used the energy and enstrophy balance~\eqref{eq:ENERbal},~\eqref{eq:ENSTRObal} resp., for the last equality. Now the result follows using the assumptions.

\subsection{Proof of~\eqref{eq:necessary3}}
	%	This is going to be real fun so I save it for later ... :p 
	This time, we start from identity~\eqref{eq:Sparati}:
	\begin{equation}
	\frac{\bar{D}_{||}(\ell)}{\ell^3}= - \frac{4\nu\bar{\Gamma}_{||}^\prime(\ell)}{\ell^3} +\frac{2}{\ell^6}\int_0^\ell r^2 \bar{D}(r)\dr + \frac{4\alpha}{\ell^6} \int_0^\ell r^3 \bar{\da}_{||}(r)\dee r - \frac{4}{\ell^6}\int_0^\ell r^3 \bar{a}_{||}(r)\dr. \label{eq:NecBal3}
	\end{equation}
	which follows from Lemma~\ref{lem:khmparallel}. 
	
\smallskip

\noindent \textbf{Step 1.} For the first term in \eqref{eq:NecBal3}, we again use Taylor expansion, 
	\begin{equation*}
	\frac{4\nu}{\ell^3} \bar{\Gamma}_{||}'(\ell) = \frac{4\nu}{\ell^3} \bar{\Gamma}_{||}'(0) + \frac{4\nu}{\ell^2} \bar{\Gamma}_{||}''(0) + \frac{2\nu}{\ell} \bar{\Gamma}_{||}'''(0) + \frac{2\nu}{3} \bar{\Gamma}_{||}''''(\vartheta),
	\end{equation*}
	for some $\vartheta\in [0,\ell]$. From~\eqref{eq:HHHHH}, we obtain $\bar{\Gamma}_{||}'(0)=0$. Next, 
	\begin{equation*}
	%\begin{split}
	\bar{\Gamma}_{||}''(r)
	%& = \sum_{i,j,k,m}\EE\fint_{\S^{1}}\fint_{\TT^2_\lambda} n^i n^j n^k n^m u^i(x) \partial_{x_k}\partial_{x_m} u^j(x+ rn) \dx \dee n\\&
	= -\sum_{i,j,k,m}\EE\fint_{\S}\fint_{\TT^2_\lambda} n^i n^j n^k n^m \partial_{x_k} u^i(x) \partial_{x_m} u^j(x+ rn) \dx \dee n,
	%\end{split}
	\end{equation*}
	after integrating by parts. Therefore $|\bar{\Gamma}_{||}''(r)|\lesssim \EE \norm{\Grad u}^2_\lambda$ and %uniformly in $r>0$ for fixed $\nu, \damp>0$ and
	\begin{equation*}
	\bar{\Gamma}_{||}''(0) = -\frac{1}{8}\E\norm{\Grad u}_\lambda^2,
	%	\EE\fint_{\TT^2_\lambda}|\Grad u|^2 \dx,
	\end{equation*}
where we used~\eqref{eq:trigo2} again. Next, we compute $\bar{\Gamma}_{||}'''(0)$: 
% . We have
% 	\begin{equation*}
% 	\begin{split}
% 	|\bar{\Gamma}_{||}'''(r)|
% 	& = \left|\sum_{i,j,k,m,p}\EE\fint_{\S}\fint_{\TT^2_\lambda} n^i n^j n^k n^m n^p \partial_{x_k} u^i(x) \partial_{x_p}\partial_{x_m} u^j(x+ rn) \dx \dee n\right|\\
% 	&\lesssim \left(\EE\norm{\Grad u}^2_\lambda\right)^{1/2}\left(\EE \norm{\Grad \omega}^2_\lambda\right),
% 	\end{split}
% 	\end{equation*}
% 	uniformly in $r>0$, hence
	\begin{equation*}
	\bar{\Gamma}_{||}'''(0) = -\sum_{i,j,k,m,p}\EE\fint_{\S} n^i n^j n^k n^m n^p\dee n\fint_{\TT^2_\lambda} \partial_{x_k} u^i(x) \partial_{x_p}\partial_{x_m} u^j(x) \dx=0, 
	\end{equation*}
	since
	\begin{equation*}
	\fint_{\S}n^i n^j n^k n^m n^p \dee n = 0.
	\end{equation*}
	Moreover,
	\begin{equation*}
	\begin{split}
	\bar{\Gamma}_{||}''''(r) 
	& = \sum_{i,j,k,m,p,q}\EE\fint_{\S}\fint_{\TT^2_\lambda} n^i n^j n^k n^m n^p n^q \partial_{x_q}\partial_{x_k} u^i(x) \partial_{x_p}\partial_{x_m} u^j(x+ rn) \dx \dee n\\
	& = \sum_{i,j,k,m,p,q}\EE\fint_{\S} n^i n^j n^k n^m n^p n^q \fint_{\TT^2_\lambda}\partial_{x_q}\partial_{x_k} u^i(x) \partial_{x_p}\partial_{x_m} u^j(x) \dx \dee n\\
	& \quad + \sum_{i,j,k,m,p,q}\EE\fint_{\S} n^i n^j n^k n^m n^p n^q \fint_{\TT^2_\lambda}\partial_{x_q}\partial_{x_k} u^i(x) \partial_{x_p}\partial_{x_m} \delta_{rn}u^j(x) \dx \dee n\\
	& =: I_1 + I_2.
	\end{split}
	\end{equation*}
	The very last term is bounded by
	\begin{equation*}
	\left|I_2\right|\lesssim \left(\EE\norm{\Grad\omega}_\lambda^2\right)^{1/2}\left(\sup_{|h|\leq r}\EE\norm{\Grad\delta_h\omega}_\lambda^2\right)^{1/2}.
	\end{equation*}
	Inserting the expression for a sixth order isotropic tensor in the Appendix~\ref{sec:tensorshit} (and using that $u$ is divergence free), we obtain
	\begin{equation*}
	I_1 = \frac{1}{16}\EE\norm{\Grad\omega}^2_\lambda.
	\end{equation*} 
	%	(i got $1/16$ as the constant but I believe that is wrong)
	Combining, the last few calculations, we get
	\begin{equation}\label{eq:Hell2}
	\frac{4\nu}{\ell^3} \bar{\Gamma}_{||}'(\ell) =  -\frac{\nu}{2\ell^2}\EE\norm{\omega}^2_\lambda +  \frac{\nu}{24}\EE\norm{\Grad\omega}^2_\lambda + \mathcal{O}\left(\left(\eta\nu\sup_{|h|\leq \ell}\EE\norm{\Grad\delta_h\omega}^2_\lambda\right)^{\frac{1}{2}}\right).
	\end{equation}
	
\smallskip

\noindent \textbf{Step 2.} By~\eqref{eq:necessary2}, we have $r^{-3} \bar{D}(r)\to 0$ uniformly in $\damp$ and $\nu$ and therefore
	\begin{equation}
	\label{eq:easyterm}
	\frac{2}{\ell^6}\int_0^\ell r^2 \bar{D}(r)\dr \to 0,\quad \text{uniformly in }\, \damp,\nu\in (0,1).
	\end{equation}
	
\smallskip

\noindent \textbf{Step 3.} We continue to estimate
	\begin{equation*}
	\begin{split}
	\frac{4\alpha}{\ell^6} \int_0^\ell r^3 \bar{\da}_{||}(r)\dee r
	& = \frac{4\damp}{\ell^6}\int_0^\ell r^3(\bar{\da}_{||}(0)+ r\bar{\da}_{||}'(0)+\frac{r^2}{2}\bar{\da}_{||}''(\vartheta_r))\dee r\\
	& = \frac{\damp}{\ell^2}\bar{\da}_{||}(0)+ \frac{4\damp}{5\ell}\bar{\da}_{||}'(0)+\frac{2\damp}{\ell^6}\int_0^\ell r^5\bar{\da}_{||}''(\vartheta_r))\dee r.
	\end{split}
	\end{equation*}
 As above, (c.f.~\eqref{eq:ExpDampLind34}), 
	\begin{equation*}
	\frac{\damp}{\ell^2}\bar{\da}_{||}(0) = \frac{\damp}{2\ell^2}\EE\norm{(-\Delta)^{-\gamma} u}^2_\lambda,
	\end{equation*}
	and $\bar{\da}_{||}'(0)=0$. Hence, it remains to compute $\bar{\da}_{||}''(\vartheta_r)$. We have
	\begin{equation*}
	\begin{split}
	\bar{\da}_{||}''(r)& = -\EE\fint_{\S}\fint_{\TT_{\lambda}^2} n^i n^j n^k n^m (-\Delta)^{-\gamma}\partial_{x_k} u^i(x)(-\Delta)^{-\gamma}\partial_{x_m} u^j(x+rn) \dx\dee n\\
	& = -\EE\fint_{\S}\fint_{\TT_{\lambda}^2} n^i n^j n^k n^m (-\Delta)^{-\gamma}\partial_{x_k} u^i(x)(-\Delta)^{-\gamma}\partial_{x_m} u^j(x) \dx\dee n\\
	&\quad  -\EE\fint_{\S}\fint_{\TT_{\lambda}^2} n^i n^j n^k n^m (-\Delta)^{-\gamma}\partial_{x_k} u^i(x)(-\Delta)^{-\gamma}\partial_{x_m}\delta_{rn} u^j(x) \dx\dee n\\
	&=: I_1 + I_2.
	\end{split}
	\end{equation*}
	The first term in the last equality is, using~\eqref{eq:trigo2},
	\begin{equation*}
	I_1 = -\frac{1}{8}\EE\norm{(-\Delta)^{-\gamma}\omega}_\lambda^2,
	\end{equation*}
	whereas the second term in the last equality can be bounded by
	\begin{equation*}
	I_2 \leq \left(\EE\norm{(-\Delta)^{-\gamma}\omega}_\lambda^2\right)^{1/2}\left(\sup_{|h|\leq r} \EE\norm{(-\Delta)^{-\gamma}\delta_h \omega}_\lambda^2\right)^{1/2}.
	\end{equation*}
	Combining, we get
	\begin{equation}
	\label{eq:notsobad}
	\begin{aligned}
	\frac{4\alpha}{\ell^6} \int_0^\ell r^3 \bar{\da}_{||}(r)\dee r & = \frac{\damp}{2\ell^2}\EE\norm{(-\Delta)^{-\gamma} u }^2_\lambda- \frac{\damp}{24} \EE\norm{(-\Delta)^{-\gamma}\omega}^2_\lambda\\
	& \quad + \mathcal{O}\left(\left(\eta\damp\sup_{|h|\leq \ell}\EE\norm{(-\Delta)^{-\gamma}\delta_{h}\omega}^2_\lambda\right)^{1/2}\right).
	\end{aligned}
	\end{equation}
	
\smallskip

\noindent \textbf{Step 4.} We combine~\eqref{eq:Hell2},~\eqref{eq:easyterm},~\eqref{eq:notsobad} and~\eqref{eq:InputExLind3} with~\eqref{eq:NecBal3} to obtain
	\begin{equation*}
	\begin{split}
	\frac{\bar{D}_{||}(\ell)}{\ell^3}&= \frac{1}{2\ell}\left(\nu \EE\norm{\omega}^2_\lambda+\damp\EE\norm{(-\Delta)^{-\gamma} u }^2_\lambda-\eps\right)+\frac{1}{24}\left(\eta -  \nu\EE\norm{\Grad\omega}^2_\lambda - \damp \EE\norm{(-\Delta)^{-\gamma}\omega}^2_\lambda \right)\\
	&\phantom{=}+ \mathcal{O}\left(\left(\eta\nu\sup_{|h|\leq \ell}\EE\norm{\Grad\delta_h\omega}^2_\lambda+\eta\damp\sup_{|h|\leq \ell}\EE\norm{(-\Delta)^{-\gamma}\delta_{h}\omega}^2_\lambda\right)^{\frac{1}{2}}\right) + o_{\ell\to 0}(1),
	\end{split}
	\end{equation*}
	which goes to zero as $\ell\to 0$ by the assumptions.% At least after I fix the tensor calculus.

\subsection{Proof of~\eqref{eq:necessary4}}
	From Lemma~\ref{lem:S0-Sphere}, we have
	\begin{multline*}
	\frac{\bar{D}(\ell)}{\ell} = 2\nu\E\fint_{|y|\leq \ell}\!\! \fint_{\TT^2_\lambda}\Grad u(x+y)\cdot\Grad u(x)\dx\dy\\
	+ 2\damp \E\fint_{|y|\leq \ell}\!\!\fint_{\TT^2_\lambda} (-\Delta)^{-\gamma}u(x+y) \cdot(-\Delta)^{-\gamma}u(x)\dx\dy - 2\fint_{|y|\leq \ell} \tr a(y)\dy. 
	\end{multline*}
	For the first term, we use the assumption in \eqref{def:EQdiss} and apply Lemma~\ref{lem:LrgCan1} with $f^\lambda = \sqrt{\nu}\Grad u$ to deduce that this term goes to zero as $\ell\to\infty$. Similarly, the second term goes to zero as $\ell\to 0$ by the assumption in \eqref{def:EQdamp} and Lemma~\ref{lem:LrgCan1} with $f^\lambda = \sqrt{\damp}(-\Delta)^{-\gamma}u$. The last term also vanishes using Assumption~\ref{cond:g} and Lemma \ref{lem:LrgCan1}. 

\subsection{Proof of~\eqref{eq:necessary5}}
	From Lemma~\ref{lem:khmparallel}, we have
	\begin{align}
	\frac{\bar{D}_{||}(\ell)}{\ell}= - \frac{4\nu \bar{\Gamma}_{||}'(\ell)}{\ell} +\frac{2}{\ell^4}\int_0^\ell r^2 \bar{D}(r)\dr + \frac{4\alpha}{\ell^4} \int_0^\ell r^3 \bar{\da}_{||}(r)\dee r - \frac{4}{\ell^4}\int_0^\ell r^3 \bar{a}_{||}(r)\dr. \label{eq:NecBalLast}
\end{align}
	We observe that we can write the first term equivalently as
	\begin{equation*}
	\begin{split}
	- \frac{4\nu \bar{\Gamma}_{||}'(\ell)}{\ell} &=-\frac{2\nu}{\ell^2}\sum_{i,j}\fint_{\{|y|\leq\ell\}}y^i y^j\,\Delta \Gamma^{ij}(y) \dy\\
	& = \frac{2\nu}{\ell^2}\sum_{i,j,k}\fint_{\{|y|\leq\ell\}} y^i y^j\, \E\fint_{\TT^2_\lambda}\partial_{x_k} u^i(x) \partial_{x_k} u^j(x+y) \dx\dy.
	\end{split}
	\end{equation*}
	By the assumption in \eqref{def:EQdiss}, we can apply Lemma~\ref{lem:LrgCan2} for $f^\lambda =\sqrt{\nu}\partial_{x_k} u$, $k=1,2$, to see that this term vanishes as $\ell\to \infty$.
	
Next, the second term in \eqref{eq:NecBalLast} vanishes by~\eqref{eq:necessary4}.
The third and the fourth terms in \eqref{eq:NecBalLast} go to zero using \eqref{def:EQdamp}, Assumption~\ref{cond:g} (respectively) and Lemma~\ref{lem:LrgCan2}, similar to the proof of~\eqref{eq:necessary4}.

\section{Isolated cascades}

\subsection{Isolated direct cascade}
In this section we prove Theorem~\ref{thm:NonUniIsoDirect}. The rigorous formulation of the scaling laws therein is the same as that in Theorem \ref{thm:direct}.
\begin{proof}
Consider first the proof of \eqref{eq:YagIso}, which is rigorously formulated as in \eqref{eq:Yaglom1} (with $\eta$ replaced with $\eta^\ast_\nu$ and no limit in $\alpha$).
The proof proceeds as in Section \ref{sec:PrfYag1} except for the estimate on the contribution of the damping term in \eqref{eq:YagBal}, which we prove satisfies:
\begin{equation}
\lim_{\ell \to 0}\sup_{\nu \in (0,1)} \abs{2\damp \fint_{|y|\leq\ell} \mathfrak{G}(y)\dy - (2\eta - 2\eta^\ast_\nu) } = 0.
\end{equation}
We expand this term in the limit as $\ell \to 0$ as the following:
\begin{align*}
2\damp\fint_{|y|\leq\ell} \mathfrak{G}(y)\dy = 2\damp\fint_{|y|\leq\ell} \left(\mathfrak{G}(y) - \mathfrak{G}(0)\right)\dy + 2\damp \EE \norm{(-\Delta)^{-\gamma} \omega }^2_\lambda.
\end{align*}
By definition, \eqref{def:HnegPreCom} shows that the first term vanishes in the desired manner, and the latter term is $2\eta - 2\eta^\ast_\nu$ also by definition.
This completes the proof of \eqref{eq:YagIso}. 

Consider next the proof of \eqref{eq:Lind1Iso}, which is rigorously formulated as in \eqref{eq:Lindborg1} (with $\eta$ replaced with $\eta^\ast_\nu$ and no limit in $\alpha$).
As in Section \ref{sub:Lindborg1}, we start from \eqref{eq:Lindborg1bal}. The only change is the treatment of the damping term.
Beginning from \eqref{eq:DampLind1} and differentiating, we see
\begin{align}
\bar{\da}''(\ell) = -2\sum_{i,j,k} \EE\fint_{\S}\fint_{\mathbb T_\lambda^2} n^i n^k  \partial_{x_i} (-\Delta)^{-\gamma} u^j(x+\ell n) (-\Delta)^{-\gamma} \partial_{x_k} u^j(x) \dx \dee n.
\end{align}
%Observe that by $\fint_{\S^1} n^i n^k \dee n = \frac$
By \eqref{eq:trigo}, we have  
\begin{align}
\bar{\da}''(0) = -\sum_{k,j} \EE\fint_{\mathbb T_\lambda^2}  \partial_{x_k} (-\Delta)^{-\gamma} u^j(x) (-\Delta)^{-\gamma} \partial_{x_k} u^j(x) \dx  = -\E\norm{(-\Delta)^{-\gamma}\omega}_{\lambda}^2.
\end{align}
By Taylor's theorem,
\begin{align}
\bar{\da}(r)-\bar{\da}(0) = \frac{1}{2} \bar{\da}''(r_\ast) r^2 = \frac{1}{2} \bar{\da}''(0) r^2 + \frac{1}{2} \left(\bar{\da}''(r_\ast)  - \bar{\da}''(0) \right) r^2,
\end{align}
It follows by \eqref{def:HnegPreCom} that 
\begin{align}
\lim_{\ell \to 0}\sup_{\nu \in (0,1)}\abs{\frac{4\damp}{\ell^4} \int_0^\ell r \left( \bar{\da}(r)-\bar{\da}(0) \right) \dee r + \frac{\alpha}{2}\E\norm{(-\Delta)^{-\gamma} \omega}_{\lambda}^2} = 0. 
\end{align}
After combining this observation with the rest of the arguments in Section \ref{sub:Lindborg1} following \eqref{eq:Lindborg1bal}, this completes the proof of \eqref{eq:Lind1Iso}. The proof of~\eqref{eq:Lind2Iso} (again, rigorously formulated as \eqref{eq:Lindborg3} without $\alpha$ and $\eta$ replaced with $\eta^\ast_\nu$) now follows from \eqref{eq:Lind1Iso} in a manner analogous to the direct cascade in Section \ref{sub:Lindborg1}. The argument is omitted for the sake of brevity as it is essentially the same. 
\end{proof}

\subsection{Isolated inverse cascade}
In this section we prove Theorem \ref{thm:NonUniIsoInv}. The rigorous formulation of the scaling laws therein is the same as that in Theorem \ref{thm:inverse}, and hence we will refer to statements therein. 
\begin{proof}
As in the proof of Theorem \ref{thm:inverse} in Section \ref{sec:Inverse}, we begin with the proof of \eqref{eq:34invCas1Iso} (rigorously formulated analogously to \eqref{eq:34invCas1} but with no $\nu$ limit and $\eps$ replaced with $\eps^\ast_\alpha$).
This begins with the balance \eqref{eq:Lindborg1bal2}.
The term involving $\bar{a}$ is treated as in Section \ref{sec:Inverse}. 
Note that the dissipation term due to viscosity can be written as
\begin{align}
-\frac{4\nu\bar{\Gamma}^\prime(\ell)}{\ell} = 2 \nu\sum_{i,j} \E\fint_{\abs{y} \leq \ell} \fint_{\TT^2_\lambda}\partial_{x_i}u^j(x+y) \partial_{x_i}u^j(x) \dx \dy. 
\end{align}
The required vanishing of this term then follows from the assumption in \eqref{ineq:lowfreqcomp} together with Lemma~\ref{lem:LrgCan1}.

To estimate the term associated with the large-scale damping, we write (as in Section \ref{sec:Inverse}),
\begin{align*}
\frac{4\damp}{\ell^2} \int_0^\ell r\bar{\da}(r)\dr= \frac{4\damp}{\ell^2} \int_0^\ell r(\bar{\da}(r)-\bar{\da}(0))dr+2 \damp \EE \norm{(-\Delta)^{-\gamma} u}_{\lambda}^2. 
\end{align*}
The first term is treated as in Section \ref{sec:Inverse}; indeed: 
\begin{align}
\abs{\frac{4\damp}{\ell^2} \int_0^\ell r(\bar{\da}(r)-\bar{\da}(0)) \dr} \lesssim \alpha \ell \left(\EE \norm{\grad (-\Delta)^{-\gamma} u}_{\lambda}^2\right)^{1/2} \left(\EE \norm{(-\Delta)^{-\gamma} u}_{\lambda}^2 \right)^{1/2}. 
\end{align}
Note that, since $\gamma \geq 0$, then for $\theta = \frac{\gamma}{1+\gamma}, $
\begin{align}
\EE \norm{\grad (-\Delta)^{-\gamma} u}_\lambda^2 \leq \left(\EE \norm{(-\Delta)^{-\gamma} u}_{\lambda}^{2}\right)^{\theta} \left(\EE \norm{\grad u}_{\lambda}^{2}\right)^{1-\theta},
\end{align}
and hence, 
\begin{align*}
\abs{\frac{4\damp}{\ell^2} \int_0^\ell r(\bar{\da}(r)-\bar{\da}(0))dr} \lesssim \ell \alpha^{1- \frac{1 + \theta}{2}} \left(\EE \norm{\grad u}_{\lambda}^{2}\right)^{\frac{1-\theta}{2}} \left(\alpha \EE \norm{(-\Delta)^{-\gamma} u}_{\lambda}^{2}\right)^{\frac{1 + \theta}{2}}  \lesssim \ell \alpha^{1- \frac{1 + \theta}{2}}. 
\end{align*}
Since $\frac{1+\theta}{2} < 1$, it follows that we can choose $\ell_\alpha = o(\alpha^{-1 + \frac{1+\theta}{2}})$ for this term to vanish.
The proof is then complete by the definition of $\eps^\ast_\alpha$. 
\end{proof}

\appendix
\section{Isotropic sixth order tensors}\label{sec:tensorshit}
We need the following lemma in Section \ref{sec:Nec} in order to provide high order expansions in the energy balance. 

\begin{lemma}[Expression for an isotropic sixth order tensor]
	We have
	\begin{align*}
	\fint_{\S}n^i n^j n^k n^m n^p n^q \,\dee n= \frac{1}{48}\big(& \delta_{i,j}\delta_{k,m}\delta_{p,q} + \delta_{i,j}\delta_{k,p}\delta_{m,q}+ \delta_{i,j}\delta_{k,q}\delta_{p,m} +	\delta_{i,k}\delta_{j,m}\delta_{p,q} + \delta_{i,k}\delta_{j,p}\delta_{m,q}\\
	&+ \delta_{i,k}\delta_{j,q}\delta_{p,m} +\delta_{i,m}\delta_{k,j}\delta_{p,q} + \delta_{i,m}\delta_{k,p}\delta_{j,q}+\delta_{i,m}\delta_{k,q}\delta_{p,j}+\delta_{i,p}\delta_{k,m}\delta_{j,q}\\
	&+ \delta_{i,p}\delta_{k,j}\delta_{m,q}+ \delta_{i,p}\delta_{k,q}\delta_{j,m}+
	\delta_{i,q}\delta_{k,m}\delta_{p,j} + \delta_{i,q}\delta_{k,p}\delta_{m,j}+ \delta_{i,q}\delta_{k,j}\delta_{p,m}	\big).
	\end{align*}
\end{lemma}	
\begin{proof}
	The left hand side is an isotropic sixth order tensor. From~\cite{Kearsley1975}, we know that it is a linear combination of 15 fundamental isotropic tensors of the form $ \delta_{i,j}\delta_{k,m}\delta_{p,q}$ and all permutations of $i,j,k,m,p,q$ in this expression. Since $i,j,k,m,p,q$ are interchangeable in $\fint_{\S}n^i n^j n^k n^m n^p n^q \,\dee n$, they must all occur with the same factor, and therefore
	\begin{align*}
	\fint_{\S}n^i n^j n^k n^m n^p n^q \,\dee n= \kappa\big(& \delta_{i,j}\delta_{k,m}\delta_{p,q} + \delta_{i,j}\delta_{k,p}\delta_{m,q}+ \delta_{i,j}\delta_{k,q}\delta_{p,m} +	\delta_{i,k}\delta_{j,m}\delta_{p,q} + \delta_{i,k}\delta_{j,p}\delta_{m,q}\\
	&+ \delta_{i,k}\delta_{j,q}\delta_{p,m} +\delta_{i,m}\delta_{k,j}\delta_{p,q} + \delta_{i,m}\delta_{k,p}\delta_{j,q}+\delta_{i,m}\delta_{k,q}\delta_{p,j}+\delta_{i,p}\delta_{k,m}\delta_{j,q}\\
	&+ \delta_{i,p}\delta_{k,j}\delta_{m,q}+ \delta_{i,p}\delta_{k,q}\delta_{j,m}+
	\delta_{i,q}\delta_{k,m}\delta_{p,j} + \delta_{i,q}\delta_{k,p}\delta_{m,j}+ \delta_{i,q}\delta_{k,j}\delta_{p,m}	\big),
	\end{align*}
	for some constant $\kappa\in\R$. It remains to compute $\kappa$. 
	We have for $i=j=k=m=p=q$
	\begin{equation*}
	\fint_{\S} (n^i)^6 \dee n = \frac{1}{2\pi}\int_0^{2\pi}\sin^6(\theta)\dee\theta = \frac{5}{16}.
	\end{equation*}
	In this case, none of the 15 terms vanishes and therefore $15\kappa = \frac{5}{16}$ and hence $\kappa=\frac{1}{48}$.
	
\end{proof}

\section*{Funding and conflict of interest}

\subsection*{Funding} 
J.B. was supported by NSF CAREER grant DMS-1552826 and NSF RNMS 1107444 (Ki-Net).
S.P-S.  was supported by NSF DMS-1803481.

\subsection*{Conflict of Interest}
The authors declare that they have no conflict of interest.

\bibliographystyle{abbrv}
\bibliography{bibliography}

\end{document}